\documentclass[a4paper,11pt,reqno]{article}
\usepackage[centering,top=2cm,bottom=2cm,right=2cm,left=2cm,
           headsep=20pt,headheight=20pt]{geometry}
\usepackage{amsmath,amssymb,amsthm,mathtools}
\usepackage{graphicx}
\usepackage{enumitem}
\usepackage[hyperindex]{hyperref}
\hypersetup{pdftitle={Additive Mertens Sum},
pdfauthor={Daoyi Peng \& Hao Liu},colorlinks,linkcolor=blue,citecolor=red,urlcolor=magenta,anchorcolor=green}
\usepackage{xcolor}
\usepackage{booktabs}
\usepackage{multirow}
\usepackage{float}
\usepackage{pgfplots}
\pgfplotsset{compat=1.18}
\usepackage{indentfirst}

\newtheoremstyle{cnplain}
  {\topsep}{\topsep}
  {\normalfont}
  {2em}
  {\bfseries}
  {.}
  {0.5em}
  {}
\theoremstyle{cnplain}
\newtheorem{theorem}{Theorem}[section]
\newtheorem{lemma}[theorem]{Lemma}
\newtheorem{proposition}[theorem]{Proposition}
\newtheorem{corollary}[theorem]{Corollary}
\newtheorem{remark}[theorem]{Remark}
\newtheorem{definition}[theorem]{Definition}
\numberwithin{equation}{section}

\makeatletter
\renewenvironment{proof}[1][\proofname]{\par\addvspace{\topsep}\pushQED{\qed}
  \setlength{\parindent}{2em}\indent{\normalfont\bfseries #1}\hskip0.5em\relax\normalfont\ignorespaces}{\popQED\par\addvspace{\topsep}}
\makeatother

\allowdisplaybreaks[4]

\begin{document}

\title{\bfseries Complete Asymptotic Expansion of the Additive Mertens Sum $S_k(x)$}
\author{Daoyi Peng \and Hao Liu}
\date{September 20, 2026}
\maketitle

\renewcommand{\thefootnote}{}
\begin{NoHyper}
\end{NoHyper}
\renewcommand{\thefootnote}{\arabic{footnote}}

\begin{abstract}
Let $p_1, \dotsc, p_k$ be primes not exceeding $x$ ($k \geqslant 2$), and define the additive Mertens sum
\[
 S_k(x) = \sum_{p_1 \leqslant x} \cdots \sum_{p_k \leqslant x} \frac{1}{p_1 + \dotsm + p_k}.
\]
In contrast to Tenenbaum's generalized (multiplicative) Mertens sum, whose leading term has order $(\log \log x)^k$, the sum $S_k(x)$ has leading term of order $x^{k-1}/\log^k x$. We establish the complete asymptotic expansion
\[
S_k(x) = \frac{x^{k-1}}{\log^k x} \sum_{n=0}^{N} \frac{E_{k,n}}{\log^n x} + O_{k,N}\left(\frac{x^{k-1}}{\log^{k+N+1} x}\right) \quad (\forall\, N \geqslant 0),
\]
where the coefficients are given by absolutely convergent multiple logarithmic integrals
\[
E_{k,n} = (-1)^n \int_{(0,1]^k} \frac{h_n(\log t_1, \dotsc, \log t_k)}{t_1 + \dotsm + t_k}\, \mathrm{d}\mathbf{t},
\]
with $h_n$ the complete homogeneous symmetric polynomial of degree $n$. We derive closed-form expressions for the first two coefficients $E_{k,0}$ and $E_{k,1}$ for all $k$, and obtain the closed form for the diagonal part of the third coefficient $E_{k,2}$ (with $E_{k,2}$ fully explicit for $k \leqslant 7$), together with a uniform closed form for the coefficients of the new constants $G(-1/m)$, where $G(z)=\int_0^z \mathrm{Li}_2(t)/(1+t)\,\mathrm dt$, whose genuinely new instances first occur for $k \geqslant 5$ (Proposition~\ref{prop:Gcoefficients}); consequently, the first three terms of the expansions of $S_2(x)$ and $S_3(x)$ are fully explicit. For $k = 2$, we further obtain a closed-form expression for the entire sequence $\{E_{2,n}\}_{n \geqslant 0}$, whose values are explicit $\mathbb{Q}$-linear combinations of $\log 2$ and zeta values $\zeta(j)$. The proofs rely on the real-variable form of the prime number theorem, variable rescaling, and multivariate Taylor remainder estimates.
\end{abstract}

\section{Introduction}

\subsection{Background and the Problem}
Sums of reciprocals of primes are a classical topic in analytic number theory. Mertens' second theorem gives
\[
\sum_{p \leqslant x} \frac{1}{p} = \log \log x + c_1 + O \left(\frac{1}{\log x}\right),
\]
where $c_1 = \gamma - \sum_p \{\log(1 - 1/p)^{-1} - 1/p\}$ is the Mertens constant and $\gamma$ is the Euler--Mascheroni constant.

The multiplicative analogue of Mertens' second theorem has attracted considerable recent attention. The double case $k = 2$ goes back to Saidak \cite{Saidak02}; for the double and triple cases $k = 2, 3$, Popa \cite{Popa1, Popa2} obtained asymptotic evaluations of $\sum_{p_1 \dotsm p_k \leqslant x} 1/(p_1 \dotsm p_k)$ by elementary means, using Abel summation and the Dirichlet hyperbola method. For general $k$, Tenenbaum \cite{T2017, Ten2019} investigated the generalized (multiplicative) Mertens sum
\[
\widetilde{S}_{k}(x) := \sum_{p_1 \dotsm p_k \leqslant x} \frac{1}{p_1 \dotsm p_k} = P_k(\log \log x) + O_k\left(\frac{(\log \log x)^{k-1}}{\log x}\right),
\]
whose main term is a polynomial $P_k(X)=\sum_{j=0}^{k}\lambda_{j,k}X^{j}$ of degree $k$ in $\log \log x$, with coefficients
\[
\lambda_{j,k}=\sum_{m=0}^{k-j}\binom{k}{m,\,j,\,k-m-j}(c_1-\gamma)^{k-m-j}\left(\frac{1}{\Gamma}\right)^{(m)}(1),
\]
where $c_1$ is the Mertens constant appearing in Mertens' second theorem above; the proof uses the Selberg--Delange method \cite{TGSM}. The error term with exponent $k-1$ is taken from the corrected version \cite{Ten2019}; the published version \cite{T2017} carries the weaker exponent $k$. Subsequently, Qi and Hu \cite{QH2019} recovered and extended Tenenbaum's evaluations by elementary methods, again based on Abel summation and the Dirichlet hyperbola method. All of these works, however, remain within the multiplicative setting.

In this paper, we consider the additive Mertens sum
\begin{equation}\label{eq:Sk}
S_k(x):= \sum_{p_1 \leqslant x} \dotsm \sum_{p_k \leqslant x}\frac{1}{p_1+\dotsb+p_k} \quad (x>2),
\end{equation}
where $p_1, \dotsc, p_k$ are primes not exceeding $x$, $k \geqslant 2$. To the best of our knowledge, the complete asymptotic expansion of this additive analogue does not appear to have been studied previously. The methods available for the multiplicative case do not transfer directly, since Abel summation together with the Dirichlet hyperbola or Selberg--Delange method \cite{TGSM} relies on the product structure of the denominator, which is lost for $p_1 + \dotsb + p_k$. The key observation for the additive problem is that the AM--GM inequality, combined with the prime number theorem, isolates the correct order of magnitude $S_k(x) \asymp x^{k-1}/\log^k x$, which serves as the starting point of the finer asymptotic analysis below. It is worth noting that the order of magnitude and the leading constant are consistent with the standard circle-method heuristic (equivalently, with the prime-number-theorem main term) for the $k$-fold additive problem: the singular series has mean value $1$ in the precise summatory sense $\sum_{s \leqslant X} \mathfrak{S}_k(s) = X + O_k(1)$, while the singular integral equals $E_{k,0}$; the contribution of the present paper is therefore the \emph{complete} expansion and the explicit arithmetic structure of all its coefficients (see Remark~\ref{rem:circle}). Indeed, by the AM--GM inequality, $p_1 + \dotsb + p_k \geqslant k \sqrt[k]{p_1 \dotsm p_k}$, so
\[
S_k(x) \leqslant \frac{1}{k} \sum_{p_1 \leqslant x} \dotsm \sum_{p_k \leqslant x} \frac{1}{\sqrt[k]{p_1 \dotsm p_k}} = \frac{1}{k} \left(\sum_{p \leqslant x} p^{-1/k}\right)^k.
\]
Applying the prime number theorem and Stieltjes integration to $\sum_{p \leqslant x} p^{-1/k}$ yields
\[
\sum_{p \leqslant x} p^{-1/k} \sim \frac{x^{1-1/k}}{(1 - 1/k) \log x},
\]
and hence
\[
S_k(x) \leqslant \left(\frac{k^{k-1}}{(k-1)^k} + o(1)\right) \frac{x^{k-1}}{\log^k x}.
\]
On the other hand, since $p_1 + \cdots + p_k \leqslant kx$,
\[
S_k(x) \geqslant \frac{1}{kx} \left(\sum_{p \leqslant x} 1\right)^k = \frac{\pi(x)^k}{kx} = \frac{1 + o(1)}{k} \cdot \frac{x^{k-1}}{\log^k x}.
\]
The natural question is: what is the exact leading coefficient of $S_k(x)$? Can one obtain a finer asymptotic expansion? We provide a complete answer below.

\subsection{Main Results}

We shall see that the coefficients in the asymptotic expansion of $S_k(x)$ are naturally expressed as multiple logarithmic integrals over the unit cube.

\begin{definition}
Let $h_n$ be the complete homogeneous symmetric polynomial of degree $n$ in $k$ variables,
\[
h_n(y_1, \ldots, y_k) = \sum_{\substack{n_1 + \dotsb + n_k = n \\ n_j \geqslant 0}} \prod_{j=1}^{k} y_j^{n_j} \quad (h_0 = 1).
\]
Define
\begin{equation}\label{eq:Ekn}
E_{k,n} := (-1)^{n} \int_{(0,1]^k} \frac{h_n(\log t_1, \dotsc, \log t_k)}{t_1 + \dotsb + t_k} \, \mathrm{d}\mathbf{t} \quad (n \geqslant 0).
\end{equation}
\end{definition}

\begin{theorem}[Complete asymptotic expansion]\label{thm:main}
Let $k \geqslant 2$ be fixed. Then for any fixed integer $N \geqslant 0$, as $x \to \infty$,
\begin{equation}\label{eq:mainexp}
S_k(x) = \frac{x^{k-1}}{\log^k x} \sum_{n=0}^{N} \frac{E_{k,n}}{\log^n x}
+ O_{k,N}\Big(\frac{x^{k-1}}{\log^{k+N+1}x}\Big),
\end{equation}
where each $E_{k,n}$ is given by \eqref{eq:Ekn} and is absolutely convergent.
\end{theorem}

Theorem \ref{thm:main} reduces the problem to computing the coefficients $E_{k,n}$. The following theorem and proposition describe the arithmetic structure of the first three coefficients. Throughout, we say that a closed form has \emph{polylog order} $s$ if $s$ is the maximum order of a polylogarithm $\mathrm{Li}_s$ appearing in it, with the convention that $\log = \mathrm{Li}_1$ has order $1$; this is not the ``weight'' of multiple zeta value theory, which is additive over products of constants.

\begin{theorem}[The first two coefficients]\label{thm:coeff}
Let $H_m=\sum_{i \leqslant m}1/i$. Then:
\begin{enumerate}[label={(\roman*)},leftmargin=2.2em]
\item (Leading term, polylog order $1$)\; 
$\displaystyle E_{k,0}=\frac1{(k-1)!}\sum_{j=2}^{k}(-1)^{k+j}j^{k-1}\binom{k}{j}\log j$.
\item (Second order, polylog order $2$)\; In the basis\footnote{For convenience, we call $\{\pi^2, \log m, \mathrm{Li}_{2}(-1/m)\}$ a basis; we do not presuppose its $\mathbb{Q}$-linear independence. If linear relations exist, the coefficients below correspond to the representation under one particular choice.} $\{\pi^2, \log m, \mathrm{Li}_2(-1/m)\}$,
\[
E_{k,1}=(-1)^{k} \frac{k}{12(k-2)!}\pi^2
+\sum_{m=2}^{k} C_{\log}(k,m)\log m
+\sum_{m=2}^{k-1} (-1)^{k-m}\frac{k\,m^{k-1}}{(k-1)!} \binom{k-1}{m} \mathrm{Li}_2 \Big(-\frac{1}{m}\Big),
\]
where
\[
C_{\log}(k,m)=(-1)^{k-m} \binom{k}{m} [(k-m)\mathcal{Q}_k(m)-m \mathcal{P}_k(m-1)],
\]
and $\mathcal{P}_k$, $\mathcal{Q}_k$ are the polynomials defined in \S\ref{sec:Ek1}.
\end{enumerate}
\end{theorem}

\begin{proposition}[The third coefficient]\label{prop:Ek2}
For every $k \geqslant 2$ we have $E_{k,2}=k A_k+ \binom{k}{2} B_k$. Let $[F]$ denote the rational coefficient of the basis element $F$ in the closed form of the diagonal part $k A_k$.\footnote{Here $kA_k$ is expressed in the basis $\{\zeta(3),\pi^2\}\cup\{\mathrm{Li}_3(-1/m),\mathrm{Li}_2(-1/m): 2\leqslant m\leqslant k-1\}\cup\{\log m : 2\leqslant m\leqslant k\}$; as in the footnote to Theorem~\ref{thm:coeff}\,(ii), we do not presuppose its $\mathbb{Q}$-linear independence, and the coefficients refer to the representation under one particular choice.} This diagonal part admits a closed form for all $k$:
\[
[\zeta(3)]=(-1)^{k} \frac{3k}{2(k-2)!}, \quad
[\pi^2]=(-1)^{k} \frac{kH_{k-1}}{6(k-2)!}, \quad
\Big[\mathrm{Li}_3\Big(-\frac{1}{m}\Big)\Big]=(-1)^{k-m}\frac{2k\,m^{k-1}}{(k-1)!}\binom{k-1}{m},
\]
and $[\mathrm{Li}_2(-1/m)]=H_{k-1}[\mathrm{Li}_3(-1/m)]$ for $2 \leqslant m \leqslant k-1$. The coefficients of the logarithmic terms in $k A_k$ are given by the explicit finite sum $[\log]_A$ displayed in \S\ref{sec:Ek2}. The cross part $B_k$ is given by an absolutely convergent integral and also admits an explicit closed form for $k \leqslant 4$ (for $k=4$, within the $\mathbb{Q}$-algebra generated by $\zeta(3),\pi^2,\log m,\mathrm{Li}_2(-1/m)$ and $\mathrm{Li}_3(-1/m)$; see Proposition \ref{prop:Bk}); for $k \geqslant 5$ the new constants that occur are described by the uniform closed form of Proposition \ref{prop:Gcoefficients}, which makes $E_{k,2}$ fully explicit for $k \leqslant 7$.
\end{proposition}

As a corollary, substituting $k = 2, 3$ from Theorem \ref{thm:coeff} and Proposition \ref{prop:Ek2} into Theorem \ref{thm:main} yields the explicit form of the first three terms in the expansions of $S_2(x)$ and $S_3(x)$, recorded in Corollaries~\ref{cor:S2} and~\ref{cor:S3}.
\begin{corollary}\label{cor:S2}
As $x \to \infty$, 
\begin{equation}\label{eq:S2}
S_2(x) = \frac{x}{\log^2 x}\left[2\log 2
+\frac{\zeta(2) + 4 \log 2}{\log x}
+\frac{3 \zeta(2) + 3 \zeta(3) + 12 \log 2}{\log^2 x}
+ O\Big(\frac1{\log^3 x}\Big)\right],
\end{equation}
where $\zeta(2)=\pi^2/6$.
\end{corollary}
\begin{corollary}\label{cor:S3}
As $x \to \infty$,
\begin{equation}\label{eq:S3}
S_3(x) = \frac{x^{2}}{\log^3 x}\left[E_{3,0}+\frac{E_{3,1}}{\log x}+\frac{E_{3,2}}{\log^2 x}
+O\Big(\frac1{\log^3 x}\Big)\right],
\end{equation}
where
\begin{align*}
E_{3,0} & =\frac{9}{2} \log 3 - 6 \log 2, \\
E_{3,1} & =-\frac{\pi^2}{4} - 21 \log 2 + \frac{63}{4} \log 3 - 6 \mathrm{Li}_2 \Big(-\frac{1}{2}\Big), \\
E_{3,2} & = -\frac{41}{8} \zeta(3)-\frac{5}{4} \pi^2 - 78 \log 2 + \frac{117}{2} \log 3
        -12 \mathrm{Li}_3 \Big(-\frac{1}{2}\Big) -30 \mathrm{Li}_2 \Big(-\frac{1}{2}\Big).
\end{align*}
\end{corollary}

\subsection{Notation}

Throughout, $p$ denotes a prime, and $\pi(x) = \sum_{p \leqslant x} 1$ is the prime counting function. We write $L := \log x$. $\mathrm{Li}_s$ denotes the polylogarithm of order $s$, and $\mathrm{Li}(x) = \int_{2}^{x} \mathrm{d}t/\log t$ the logarithmic integral. $\gamma$ is the Euler–Mascheroni constant.  $\zeta(s)$ is the Riemann zeta function. We write $f \ll g$ synonymously with $f = O(g)$, and use subscripts to indicate dependence of the implied constant. $f \asymp g$ means $f \ll g \ll f$. For a finite set $S$, let $\# S$ denote its cardinality. Finally, $\mathrm{d}\mathbf{t} = \mathrm{d}t_1 \cdots \mathrm{d}t_k$.

\section{The Coefficient Integrals \texorpdfstring{$E_{k,n}$}{Ekn} and Their Convergence}\label{sec:tools}

\subsection{Absolute Convergence}
\begin{lemma}\label{lem:conv}
For all $k \geqslant 2$ and $n \geqslant 0$, the integral \eqref{eq:Ekn} is absolutely convergent, and
\[
|E_{k,n}| \leqslant I_n: = \int_{(0,1]^k}\frac{|\log(t_1 \dotsm t_k)|^{n}}{t_1+ \dotsb + t_k}\mathrm{d}\mathbf{t} < \infty.
\]
\end{lemma}
\begin{proof}
Substituting $t_j = \mathrm{e}^{-u_j}$, which maps $(0,1]^k$ bijectively onto $[0,\infty)^k$, and using
that $h_n$ is homogeneous of degree $n$, we obtain
\[
E_{k,n}=\int_{[0,\infty)^k}\frac{h_n(-u_1,\dotsc,-u_k)\,\mathrm{e}^{-(u_1+\dotsb+u_k)}}{\mathrm{e}^{-u_1}+\dotsb+\mathrm{e}^{-u_k}}\,\mathrm{d}\mathbf{u}
=(-1)^n\int_{[0,\infty)^k}\frac{h_n(u_1,\dotsc,u_k)\,\mathrm{e}^{-(u_1+\dotsb+u_k)}}{\mathrm{e}^{-u_1}+\dotsb+\mathrm{e}^{-u_k}}\,\mathrm{d}\mathbf{u},
\]
and likewise $I_n=\int_{[0,\infty)^k}(\sum_j u_j)^n\mathrm{e}^{-\sum_j u_j}(\sum_j\mathrm{e}^{-u_j})^{-1}\mathrm{d}\mathbf{u}$.
Two elementary inequalities control both integrands.

First, since $h_n(u)=\sum_{|\alpha|=n}\prod_j u_j^{\alpha_j}$ while
$(\sum_j u_j)^n=\sum_{|\alpha|=n}\binom{n}{\alpha}\prod_j u_j^{\alpha_j}$ with all multinomial
coefficients at least $1$,
\begin{equation}\label{eq:hn-bound}
h_n(u_1,\dotsc,u_k)\leqslant\Big(\sum_{j=1}^k u_j\Big)^n\qquad(u_j\geqslant0).
\end{equation}
Second, the AM--GM inequality applied to $\mathrm{e}^{-u_1},\dotsc,\mathrm{e}^{-u_k}$ gives
\[
\sum_{j=1}^k\mathrm{e}^{-u_j}\geqslant k\Big(\prod_{j=1}^k\mathrm{e}^{-u_j}\Big)^{1/k}
=k\,\mathrm{e}^{-\frac{1}{k}(u_1+\dotsb+u_k)},
\]
that is,
\begin{equation}\label{eq:exp-ratio}
\frac{\mathrm{e}^{-(u_1+\dotsb+u_k)}}{\mathrm{e}^{-u_1}+\dotsb+\mathrm{e}^{-u_k}}
\leqslant\frac{1}{k}\,\mathrm{e}^{-\frac{k-1}{k}(u_1+\dotsb+u_k)} .
\end{equation}

Consequently both integrands are dominated by
$\frac{1}{k}(\sum_j u_j)^n\mathrm{e}^{-\frac{k-1}{k}\sum_j u_j}$, and for every $c>0$,
$\mathrm{vol}\{u\geqslant0:\sum_j u_j\leqslant \sigma\}=\sigma^k/k!$ yields
\[
\int_{[0,\infty)^k}\Big(\sum_{j=1}^k u_j\Big)^n\mathrm{e}^{-c\sum_j u_j}\,\mathrm{d}\mathbf{u}
=\frac{1}{(k-1)!}\int_0^\infty \sigma^{n+k-1}\mathrm{e}^{-c\sigma}\,\mathrm{d}\sigma
=\frac{(n+k-1)!}{(k-1)!\,c^{n+k}}<\infty .
\]
Taking $c=\frac{k-1}{k}$ we conclude, for all $k\geqslant2$ and $n\geqslant0$,
\begin{equation}\label{eq:conv-bound}
|E_{k,n}|\leqslant I_n\leqslant\frac{(n+k-1)!}{k\,(k-1)!}\Big(\frac{k}{k-1}\Big)^{n+k}<\infty .
\end{equation}
\end{proof}

\subsection{Laplace Tools}

Substituting $1/(\sum_{j=1}^{k} t_j) = \int_{0}^{\infty} \mathrm{e}^{-x \sum_{j=1}^{k} t_j}\, \mathrm{d}x$ and separating by coordinates, we shall need the two single-variable integrals
\begin{equation}\label{eq:uv}
u(x)=\int_{0}^{1} \mathrm{e}^{-xt} \, \mathrm{d}t = \frac{1-\mathrm{e}^{-x}}{x}, \qquad
v(x)=\int_{0}^{1} (\log t)\mathrm{e}^{-xt}\, \mathrm{d}t = -\frac{U(x)}{x},
\end{equation}
\begin{equation}\label{eq:U}
U(x)=\int_{0}^{x} \frac{1-\mathrm{e}^{-y}}{y}\, \mathrm{d}y = \gamma + \log x + E_1(x), \qquad U'(x) = \frac{1-\mathrm{e}^{-x}}{x},
\end{equation}
where $E_1$ is the exponential integral. The Laplace transform of $U$,
\begin{equation}\label{eq:g0}
g_0(s):=\mathcal L\{U\}(s)=\frac{\log(s+1)-\log s}{s},
\end{equation}
will appear below as the right-hand side of the differential relation $\mathcal{H}_k^{(k)} = g_0$ (superscripts denote derivatives) satisfied by the generating function. We also define the forward difference
\[
\Delta^{k-1} f(0) = \sum_{j=0}^{k-1} (-1)^{k-1-j} \binom{k-1}{j} f(j),
\]
whose key property is that $\Delta^{k-1} P = 0$ for any polynomial $P$ with $\deg P \leqslant k - 2$.

\section{Proof of the Complete Asymptotic Expansion}\label{sec:asymp}

This section proves Theorem \ref{thm:main}. Fix $k \geqslant 2$ and write $L = \log x$.

\subsection{Integral Representation and Extraction of the Main Term}

Since the prime counting function $\pi(x)$ is a step function, \eqref{eq:Sk} can be written as a Stieltjes integral
\begin{equation}\label{eq:stieltjes}
S_k(x) = \int_{[2,x]^k}\frac{\mathrm{d}\pi(u_1) \dotsm \mathrm{d}\pi(u_k)}{u_1+\dotsb+u_k}.
\end{equation}
By the prime number theorem with classical error term, $\mathrm{d}\pi(u) = \mathrm{d}u/\log u + \mathrm{d}R(u)$, where
\begin{equation}\label{eq:PNT}
R(u) = \pi(u)-\mathrm{Li}(u)\ll u\, \mathrm{e}^{-c\sqrt{\log u}}\qquad(c>0).
\end{equation}
Substituting \eqref{eq:PNT} into \eqref{eq:stieltjes} and expanding the $k$-fold product gives
\begin{equation}\label{eq:split}
S_k(x) = \sum_{J\subseteq\{1, \dots,k\}} \mathcal E_J(x), \qquad
\mathcal E_J(x)=\int_{[2,x]^k}\frac{1}{\sum u}\prod_{j\in J}\mathrm{d} R(u_j)\prod_{j\notin J}\frac{\mathrm{d} u_j}{\log u_j}.
\end{equation}
The main term is the $J = \varnothing$ term $\mathcal{M}_k(x) := \mathcal{E}_\varnothing(x)$; the remaining $\mathcal{E}_J$ ($J \neq \varnothing$) constitute the error.

\subsection{The Error Terms \texorpdfstring{$J\neq\varnothing$}{J≠∅} Are Negligible}

\begin{lemma}\label{lem:err}
For each non-empty $J$, $\mathcal{E}_J(x) \ll_k x^{k-1} \mathrm{e}^{-c'\sqrt{\log x}}$ for some $c' > 0$. Consequently, $\sum_{J \neq \varnothing} \mathcal{E}_J \ll_{k, A} x^{k-1}/\log^A x$ for all $A > 0$.
\end{lemma}

\begin{proof}
Without loss of generality, let $J = \{1, \ldots, r\}$ with $r = |J| \geqslant 1$. Treating $u_{r+1}, \ldots, u_k \in [2, x]$ as parameters, we integrate by parts successively in $u_1, \ldots, u_r$ for the inner $r$-fold Stieltjes integral. Each integration by parts transfers the differential from $R$ onto the smooth kernel $\kappa(\mathbf{u}) = (\sum u)^{-1}$, producing a boundary contribution ($R(u_i)$ evaluated at $u_i \in \{2, x\}$, with $R(2) = \pi(2) - \mathrm{Li}(2) = 1 = O(1)$ and $R(x) \ll x\, \mathrm{e}^{-c\sqrt{\log x}}$) together with an interior contribution ($\int_{2}^{x} R(u_i)\, \partial_{u_i}(\cdots)\, \mathrm{d}u_i$). After $r$ such steps, $\mathcal{E}_J$ is expressed as a sum of finitely many (at most $2^r$) terms. Writing the ``interior'' variable set of such a term as $S \subseteq J$ with $|S| = m$ and the ``boundary'' variable set as $B = J \setminus S$, the term has the form
\[
T=\pm \prod_{i\in B} R(u_i^{\ast}) \int_{[2,x]^{S\cup J^c}} \partial^{(S)}\kappa(\mathbf{u})\prod_{i \in S}R(u_i)\,\mathrm{d}u_i \prod_{j \in J^c} \frac{\mathrm{d}u_j}{\log u_j}, \qquad u_{i}^{\ast} \in \{2,x\}.
\]
Iterating $\partial_u (\sum u)^{-1} = -(\sum u)^{-2}$, the mixed partial derivative in the distinct variables of $S$ is $\partial^{(S)} \kappa = (-1)^m m! (\sum u)^{-(m+1)}$. Using $|R(u_i)| \ll u_i\, \mathrm{e}^{-c\sqrt{\log u_i}}$, $1/\log u_j \leqslant 1/\log 2$, and $\prod_{i \in S} u_i \leqslant (\sum u)^m$ (whence $\prod_{i \in S} u_i / (\sum u)^{m+1} \leqslant 1/\sum u$), we obtain
\begin{equation}\label{eq:Tbound}
|T|\ll_{r}\prod_{i\in B}|R(u_i^{\ast})|\int_{[2,x]^{S\cup J^c}}\frac{\prod_{i\in S}\mathrm{e}^{-c\sqrt{\log u_i}}}{\sum u}\prod_{i \in S \cup J^{c}} \mathrm{d} u_{i},
\end{equation}
where the free variables are those in $S \cup J^c$, totaling $\ell := m + (k - r)$. We further use the elementary estimate
\begin{equation}\label{eq:basic}
\int_{[2,x]^{\ell}}\frac{\mathrm{d}u_1 \cdots \mathrm{d}u_\ell}{u_1+\dotsb+u_\ell} \ll x^{\ell-1}\log x \qquad (\ell \geqslant 1),
\end{equation}
(for $\ell \geqslant 2$, rescaling $u = x \mathbf{t}$ gives $\ll x^{\ell-1}$; for $\ell = 1$, it is $\log(x/2)$. When $\ell = 0$, i.e.\ $S = \varnothing$ and $r = k$, no integral remains: the term reduces to $\prod_{i \in B} R(u_i^{\ast}) \cdot (\sum u^{\ast})^{-1}$, which is $\ll x^{k-1}\, \mathrm{e}^{-c\sqrt{\log x}}$ if some $u_i^{\ast} = x$, and $O(1)$ if all $u_i^{\ast} = 2$.) We now estimate \eqref{eq:Tbound} in three cases.

\textbf{(a)} $B$ contains the endpoint $x$ (some $u_i^* = x$). Then $\prod_{i \in B} |R(u_i^*)| \ll x^a\, \mathrm{e}^{-c\sqrt{\log x}}$, where $a = \#\{i \in B : u_i^* = x\}$ (factors at the endpoint $2$ contribute $O(1)$). Combined with \eqref{eq:basic} and the bound $a + m \leqslant r$, we obtain
\[
|T| \ll x^a\, \mathrm{e}^{-c\sqrt{\log x}} \cdot x^{m + k - r - 1} \log x \ll x^{k-1} \log x\, \mathrm{e}^{-c\sqrt{\log x}} \ll x^{k-1}\, \mathrm{e}^{-c''\sqrt{\log x}}.
\]

\textbf{(b)} $B$ lies entirely at the endpoint $2$ and $S \neq \varnothing$. Pick $i_0 \in S$ and split according to $u_{i_0} \in [\sqrt{x}, x]$ versus $[2, \sqrt{x}]$. For the former, $\mathrm{e}^{-c\sqrt{\log u_{i_0}}} \leqslant \mathrm{e}^{-c'\sqrt{\log x}}$, and the remaining $\ell$-dimensional integral by \eqref{eq:basic} is $\ll x^{\ell-1} \log x = x^{m + k - r - 1} \log x \leqslant x^{k-1} \log x$, so this part is $\ll x^{k-1} \mathrm{e}^{-c''\sqrt{\log x}}$. For the latter, $u_{i_0} \in [2, \sqrt{x}]$, we distinguish two sub-cases: if $\ell \geqslant 2$ (other free variables remain), use $1/\sum u \leqslant 1/\sum_{i \in (S \cup J^c) \setminus \{i_0 \}} u$ and integrate $u_{i_0}$ over an interval of length $\sqrt{x}$ (the remaining $(\ell-1)$-dimensional integral being $\ll x^{\ell-2} \log x$), giving $\ll \sqrt{x} \cdot x^{m + k - r - 2} \log x \leqslant x^{k - 3/2} \log x \ll x^{k-1} x^{-1/2} \log x$; if $\ell = 1$ (which forces $m = 1$ and $r = k$, i.e., no other free variables), then $\sum u \geqslant u_{i_0}$, and this part contributes $\int_2^{\sqrt{x}} \frac{\mathrm{e}^{-c\sqrt{\log u_{i_0}}}}{\sum u}\, \mathrm{d}u_{i_0} \leqslant \int_2^{\sqrt{x}} \frac{\mathrm{d}u_{i_0}}{u_{i_0}} \ll \log x \ll x^{k-1}$ (since $k \geqslant 2$).

\textbf{(c)} $B$ lies entirely at the endpoint $2$ and $S = \varnothing$ ($J$ entirely at the endpoint $2$). Here $\prod_{i \in B} |R(2)| = O(1)$, $m = 0$, and
\[
|T| \ll \int_{[2, x]^{k - r}} \frac{\prod_{j \in J^c} \mathrm{d}u_j}{\sum u} \ll x^{k - r - 1} \log x \leqslant x^{k-2} \log x \qquad (r \geqslant 1).
\]

All three cases yield $|T| \ll x^{k-1} \max\big(\mathrm{e}^{-c'\sqrt{\log x}},\, x^{-1/2}\big) \log x$. Since $x^{-1/2} \leqslant \mathrm{e}^{-c'\sqrt{\log x}}$ for large $x$ and $\log x \cdot \mathrm{e}^{-c'\sqrt{\log x}} \ll \mathrm{e}^{-c''\sqrt{\log x}}$, each term satisfies $|T| \ll x^{k-1}\, \mathrm{e}^{-c''\sqrt{\log x}}$. Summing over the finitely many $T$ gives $\mathcal{E}_J \ll_k x^{k-1}\, \mathrm{e}^{-c''\sqrt{\log x}}$. Finally, $\mathrm{e}^{-c''\sqrt{\log x}}$ is smaller than any $\log^{-A} x$ (for any $A > 0$), and the lemma follows. 
\end{proof}

\subsection{Rescaling of the Main Term}

In $\mathcal{M}_k(x) = \int_{[2, x]^k} \frac{1}{\sum u} \prod_j \frac{\mathrm{d}u_j}{\log u_j}$, set $u_j = x t_j$ with $t_j \in [2/x, 1]$. Then $\sum u = x \sum t$, $\mathrm{d}u_j = x\, \mathrm{d}t_j$, $\log u_j = L + \log t_j$, so
\begin{equation}\label{eq:Mscaled}
\mathcal{M}_k(x) = x^{k-1} \int_{[2/x, 1]^k} \frac{1}{\sum t} \prod_{j=1}^k \frac{1}{L + \log t_j}\, \mathrm{d}\mathbf{t}.
\end{equation}
Fix any $\delta \in (0, 1)$ (independent of $x$; the eventual implied constant depending on $k, N, \delta$), and partition the integration domain $[2/x, 1]^k = \mathcal{G}_\delta \sqcup \mathcal{B}_\delta$, where
\[
\mathcal{G}_\delta=\{\mathbf{t}: t_j \geqslant x^{-\delta}, \ \forall j\}, \qquad \mathcal{B}_\delta=[2/x,1]^{k} \setminus \mathcal{G}_\delta.
\]
(For sufficiently large $x$, $x^{-\delta} > 2/x$.)

\textbf{Boundary region $\mathcal{B}_\delta$}. Since $\mathcal{B}_\delta \subseteq \bigcup_j \{t_j < x^{-\delta}\}$, by symmetry it suffices to estimate the contribution from $\{t_1 < x^{-\delta}\}$, i.e., from $\{u_1 < x^{1 - \delta}\}$ in the $u$-variables. Using $1/\sum u \leqslant 1/(u_2 + \cdots + u_k)$, $1/\log u_j \leqslant 1/\log 2$, the elementary estimate $\int_{[2, x]^\ell} \mathrm{d}u_1 \cdots \mathrm{d}u_\ell / (u_1 + \cdots + u_\ell) \ll x^{\ell-1} \log x$ (with $\ell = k - 1$), and $\int_2^{x^{1-\delta}} \mathrm{d}u_1/\log u_1 \ll x^{1-\delta}/\log x$, this contribution is
\[
\ll_k \frac{x^{1-\delta}}{\log x} \cdot x^{k-2} \log x = x^{k-1-\delta}.
\]
The same holds for $j = 2, \ldots, k$. Hence the total contribution of $\mathcal{B}_\delta$ to $\mathcal{M}_k$ is $\ll_{k,\delta} x^{k-1-\delta}$, which is negligible (smaller than any $x^{k-1} \log^{-A} x$).

\textbf{Main region $\mathcal{G}_\delta$}. On $\mathcal{G}_\delta$, write $w_j = \log t_j / L \in [-\delta, 0]$. Then
\[
\prod_{j=1}^{k} \frac{1}{L + \log t_j} = \frac{1}{L^k}\, F(\mathbf{w}), \qquad F(\mathbf{w}) := \prod_{j=1}^{k} \frac{1}{1 + w_j}, \qquad |w_j| \leqslant \delta < 1.
\]
$F$ is smooth on $\{|w_j| \leqslant \delta\}$, and its $N$-th order Taylor polynomial at the origin is precisely
\[
T_N(\mathbf{w}) = \sum_{n=0}^{N} (-1)^{n} h_n(\mathbf{w}) = \sum_{n=0}^{N} \frac{(-1)^{n} h_n(\log \mathbf{t})}{L^n},
\]
since $\prod_j (1 + w_j)^{-1} = \sum_{n \geqslant 0} (-1)^n h_n(\mathbf{w})$ (with $h_n$ the complete homogeneous symmetric polynomial).

\begin{lemma}\label{lem:taylor}
There exists a constant $C = C(k, N, \delta)$ such that on $\mathcal{G}_\delta$,
\[
\big|F(\mathbf w)-T_N(\mathbf w)\big|\leqslant C\,h_{N+1}\big(|w_1|, \dots,|w_k|\big)
=\frac{C}{L^{N+1}}\,h_{N+1}\big(|\log t_1|, \dots,|\log t_k|\big).
\]
\end{lemma}
\begin{proof}
By the multivariate Taylor formula with integral remainder,
\[
F(\mathbf{w}) - T_N(\mathbf{w}) = \sum_{|\boldsymbol{\alpha}| = N+1} \frac{N+1}{\boldsymbol{\alpha}!}\, \mathbf{w}^{\boldsymbol{\alpha}} \int_0^1 (1 - \theta)^N \partial^{\boldsymbol{\alpha}} F(\theta \mathbf{w})\, \mathrm{d}\theta.
\]
Since $\partial^{\boldsymbol{\alpha}} F(\mathbf{w}) = \prod_j (-1)^{\alpha_j} \alpha_j! (1 + w_j)^{-1-\alpha_j}$, on $|w_j| \leqslant \delta$ we have $|\partial^{\boldsymbol{\alpha}} F| \leqslant (1 - \delta)^{-(k + N + 1)} \prod_j \alpha_j!$. Substituting yields $|F - T_N| \leqslant (1 - \delta)^{-(k + N + 1)} \sum_{|\boldsymbol{\alpha}| = N+1} |\mathbf{w}^{\boldsymbol{\alpha}}| = C\, h_{N+1}(|\mathbf{w}|)$ with $C = (1 - \delta)^{-(k + N + 1)}$. The last identity follows from $w_j = \log t_j / L$.
\end{proof}

Expanding the main-region integral according to $T_N$ gives, term by term,
\[
\frac{x^{k-1}}{L^k} \sum_{n=0}^N \frac{(-1)^n}{L^n} \int_{\mathcal{G}_\delta} \frac{h_n(\log \mathbf{t})}{\sum t}\, \mathrm{d}\mathbf{t},
\]
with remainder controlled by Lemmas \ref{lem:taylor} and \ref{lem:conv}:
\[
\frac{x^{k-1}}{L^k} \int_{\mathcal{G}_\delta} \frac{|F - T_N|}{\sum t}\, \mathrm{d}\mathbf{t} \leqslant \frac{x^{k-1}}{L^k} \cdot \frac{C}{L^{N+1}} \int_{(0, 1]^k} \frac{h_{N+1}(|\log \mathbf{t}|)}{\sum t}\, \mathrm{d}\mathbf{t} \ll \frac{x^{k-1}}{L^{k + N + 1}}.
\]

\begin{lemma}[Extension of the integration domain]\label{lem:tail}
For each fixed $n \geqslant 0$ and $\delta \in (0,1)$,
\[
\int_{\{t_1 < x^{-\delta}\} \cap (0,1]^k}
  \frac{h_n(|\log t_1|, \dotsc, |\log t_k|)}{t_1 + \dotsb + t_k}\,
  \mathrm{d}\mathbf{t}
\;\ll_{k,n,\delta}\; x^{-\delta/4}
\qquad (x \to \infty).
\]
Since $(0,1]^k \setminus \mathcal{G}_\delta \subseteq \bigcup_{j=1}^k \{t_j < x^{-\delta}\}$,
summing the above bound over the (at most) $k$ symmetric regions
$\{t_j < x^{-\delta}\}$ ($j = 1, \dotsc, k$), each identical to the
case $j = 1$ after permuting indices, and absorbing the resulting
factor of $k$ into the implied constant, we obtain, for each
$0 \leqslant n \leqslant N$,
\[
\int_{\mathcal{G}_\delta} \frac{h_n(\log \mathbf{t})}{\sum t}\,\mathrm{d}\mathbf{t}
= \int_{(0,1]^k} \frac{h_n(\log \mathbf{t})}{\sum t}\,\mathrm{d}\mathbf{t}
+ O_{k,n,\delta} \big(x^{-\delta/4}\big).
\]
\end{lemma}
\begin{proof}
By the bound $|h_n(\log \mathbf{t})| \leqslant h_n(|\log \mathbf{t}|)
\leqslant |\log(t_1 \dotsm t_k)|^n$ (Lemma~\ref{lem:conv}; recall that $\sum_j |\log t_j| = |\log(t_1 \dotsm t_k)|$ for $0 < t_j \leqslant 1$),
it suffices to estimate
\[
T := \int_{\{t_1 < x^{-\delta}\} \cap (0,1]^k}
     \frac{|\log(t_1 \dotsm t_k)|^n}{\sum t}\,\mathrm{d}\mathbf{t}.
\]
Let $\Omega_i = \{\mathbf{t} \in (0,1]^k : t_i \geqslant t_j \text{ for all } j\}$, so that
$(0,1]^k = \bigcup_{i=1}^k \Omega_i$ up to null sets and the pieces are congruent by symmetry;
we estimate $\int_{\Omega_i \cap \{t_1 < x^{-\delta}\}}$ for each $i$.

\textbf{Case 1}: $i = 1$ (the largest coordinate is $t_1$).
After the rescaling $t_j = t_1 r_j$ ($j \geqslant 2$, $r_j \in (0,1]$)
for which $\mathrm{d}t_2 \dotsm \mathrm{d}t_k = t_1^{k-1} \mathrm{d}r_2 \dotsm \mathrm{d}r_k$ and $\sum_j |\log t_j| = k|\log t_1| + \sum_{j \geqslant 2} |\log r_j|$, the constraint $t_1 < x^{-\delta}$
restricts the outer integral to $(0, x^{-\delta})$. Since $k \geqslant 2$
implies $t_1^{k-2} \leqslant 1$ on $(0,1]$, and the inner integral over
$\mathbf{r}$ is $\ll_{k,n} (|\log t_1| + 1)^n \ll_{k,n} |\log t_1|^n$
(for $t_1 \leqslant 1/2$), we obtain
\[
\int_{\Omega_1 \cap \{t_1 < x^{-\delta}\}} (\cdots)
\;\ll_{k,n}\; \int_0^{x^{-\delta}} |\log t_1|^n\,\mathrm{d}t_1
\;\ll\; x^{-\delta}(\log x)^n.
\]

\textbf{Case 2}: $i \neq 1$. By relabeling, it suffices to treat $i = 2$ (so $t_2 = \max_j t_j$); the cases $i = 3, \dotsc, k$ are identical after permuting the indices $\{2, \dotsc, k\}$. After the rescaling $t_j = t_2 r_j$ ($j \neq 2$), the constraint
$t_1 < x^{-\delta}$ becomes $t_2\, r_1 < x^{-\delta}$. We split into
two sub-cases:

\textit{Sub-case 2a: $t_2 < x^{-\delta/2}$.}
The outer integral is restricted to $(0, x^{-\delta/2})$, giving
\[
\int_0^{x^{-\delta/2}} t_2^{k-2}\,|\log t_2|^n\,\mathrm{d}t_2
\;\leqslant\; \int_0^{x^{-\delta/2}} |\log t_2|^n\,\mathrm{d}t_2
\;\ll\; x^{-\delta/2}(\log x)^n,
\]
while the inner integral over $\mathbf{r} \in (0,1]^{k-1}$ is
$O_{k,n}(1)$ (all one-dimensional factors $\int_0^1 |\log r_j|^{c_j}\,
\mathrm{d}r_j = c_j!$ are finite).

\textit{Sub-case 2b: $t_2 \geqslant x^{-\delta/2}$, hence
$r_1 < x^{-\delta/2}$.}
The outer integral $\int_0^1 t_2^{k-2}|\log t_2|^a\,\mathrm{d}t_2
= a!/(k-1)^{a+1}$ is $O_{k}(1)$. In the inner integral, the factor
corresponding to $r_1$ is restricted:
\[
\int_0^{x^{-\delta/2}} |\log r_1|^{c_1}\,\mathrm{d}r_1
\;\ll\; x^{-\delta/2}(\log x)^{c_1}
\qquad (0 \leqslant c_1 \leqslant n),
\]
while all other factors $\int_0^1 |\log r_j|^{c_j}\,\mathrm{d}r_j = c_j!$
remain $O(1)$.

In every sub-case, the contribution is
$\ll_{k,n} x^{-\delta/2}(\log x)^n$. Summing over the $k$ regions
$\Omega_i$ and using $x^{-\delta/2}(\log x)^n \ll_\delta x^{-\delta/4}$
for sufficiently large $x$ gives $T \ll_{k,n,\delta} x^{-\delta/4}$.
\end{proof}

Applying Lemma~\ref{lem:tail} to each $n = 0, \dotsc, N$, the
$O(x^{-\delta/4})$ discrepancy is absorbed into the
$O(x^{k-1}/L^{k+N+1})$ remainder (since $x^{-\delta/4} \ll L^{-A}$
for every $A > 0$). Since
$\int_{(0,1]^k} h_n(\log \mathbf{t})/\sum t\,\mathrm{d}\mathbf{t}
= (-1)^n E_{k,n}$, we obtain
\[
\mathcal{M}_k(x) = \frac{x^{k-1}}{L^k} \sum_{n=0}^{N} \frac{E_{k,n}}{L^n}
+ O_{k,N}\!\left(\frac{x^{k-1}}{L^{k+N+1}}\right).
\]
Combined with Lemma~\ref{lem:err}, Theorem~\ref{thm:main} follows.
\hfill$\qed$

\section{Closed Forms for the Coefficients I: \texorpdfstring{$E_{k,0}$}{Ek0} and \texorpdfstring{$E_{k,1}$}{Ek1}\label{sec:Ek1}}

We now compute the coefficients $E_{k, n}$. The unified strategy is: first use the Laplace representation to reduce the multiple integral to a single-variable integral; next construct a generating function satisfying a higher-order differential equation; finally extract the coefficients using the finite difference $\Delta^{k-1}$.

\subsection{The Leading Coefficient \texorpdfstring{$E_{k,0}$}{Ek0}}

We begin with part (i) of Theorem~\ref{thm:coeff}. The proof below also yields the integral representation of $E_{k,0}$ used in the sequel, which we record together with the closed form.

\begin{proposition}\label{prop:Ek0}
$\displaystyle E_{k,0}=\int_{0}^{\infty}\Big(\frac{1-\mathrm{e}^{-x}}{x}\Big)^{k}\,\mathrm{d} x
=\frac1{(k-1)!}\sum_{j=2}^{k}(-1)^{k+j}j^{k-1}\binom{k}{j} \log j. $
\end{proposition}
\begin{proof}
From \eqref{eq:uv},
\[
E_{k,0}=\int_{(0,1]^k}\frac{\mathrm{d}\mathbf{t}}{\sum t}=\int_0^\infty u(x)^k \, \mathrm{d}x = \int_{0}^{\infty} \frac{(1-\mathrm{e}^{-x})^k}{x^k}\, \mathrm{d}x.
\]
Let $f(x) = (1 - \mathrm{e}^{-x})^k = \sum_{j=0}^{k} \binom{k}{j} (-1)^j \mathrm{e}^{-jx}$. Performing $k - 1$ successive integrations by parts on $f(x)/x^k$: since $\lim_{x \to 0^+} f(x)/x^{k-1} = \lim_{x \to \infty} f(x)/x^{k-1} = 0$, and more generally $f^{(m)}(0) = f^{(m)}(\infty) = 0$ for $m \leqslant k - 2$ (as $f$ has a zero of order $k$ at $0$ and decays exponentially at $\infty$), all boundary terms vanish, giving
\[
\int_{0}^{\infty} \frac{f(x)}{x^k}\, dx = \frac{1}{(k-1)!} \int_{0}^{\infty} \frac{f^{(k-1)}(x)}{x}\, \mathrm{d}x.
\]
Now $f^{(k-1)}(x) = \sum_{j=0}^k \binom{k}{j} (-1)^j (-j)^{k-1} \mathrm{e}^{-jx} = \sum_{j=1}^{k} \binom{k}{j} (-1)^{k + j - 1} j^{k-1} \mathrm{e}^{-jx}$ (the $j = 0$ term contributing zero). Setting $a_j = (-1)^{k + j - 1} j^{k-1} \binom{k}{j} / (k-1)!$ for $1 \leqslant j \leqslant k$,
\[
E_{k, 0} = \int_{0}^{\infty} \sum_{j=1}^k a_j\, \frac{\mathrm{e}^{-jx}}{x}\, \mathrm{d}x.
\]
The individual integral $\int_{0}^{\infty} \mathrm{e}^{-jx}/x \, \mathrm{d}x$ diverges at $x = 0$, but by the Stirling number identity of the second kind \cite{Stanley},
\[
\frac{1}{(k-1)!} \sum_{j=1}^k (-1)^{k-j} \binom{k}{j} j^{k-1} = k\,S(k-1, k) = 0 \qquad (k \geqslant 2),
\]
i.e., $\sum_{j=1}^k a_j = 0$. Thus $\sum_{j=1}^k a_j \mathrm{e}^{-jx}/x$ has no singularity at $x = 0$, and for any $A > 0$,
\[
\sum_{j=1}^{k} a_j \frac{\mathrm{e}^{-jx}}{x} = \sum_{j=1}^{k} a_j \frac{\mathrm{e}^{-jx} - \mathrm{e}^{-Ax}}{x}
\]
(since $\sum_j a_j = 0$). By Frullani's integral \cite{Frullani}, $\int_{0}^{\infty} (\mathrm{e}^{-jx} - \mathrm{e}^{-Ax})/x\, \mathrm{d}x = \log A - \log j$, so (with the $\log A$ terms cancelling, again since $\sum_j a_j = 0$)
\[
E_{k, 0} = \sum_{j=1}^{k} a_j (\log A - \log j) = -\sum_{j=1}^k a_j \log j = \frac{1}{(k-1)!} \sum_{j=2}^k (-1)^{k+j} j^{k-1} \binom{k}{j} \log j,
\]
where the $j = 1$ term vanishes since $\log 1 = 0$. We refer to \cite{PML2019} for further details.
\end{proof}

\begin{remark}[The leading term and the circle method]\label{rem:circle}
The discussion in this remark is a heuristic comparison with the circle method and is not used in
any of the proofs below. The leading constant of Theorem~\ref{thm:main} is exactly the constant produced by the classical
circle-method heuristic for the $k$-fold additive problem. To describe the heuristic precisely, write
$r_k(s;x)=\#\{(p_1,\dotsc,p_k): p_i \leqslant x,\ p_1+\dotsb+p_k=s\}$, so that
$S_k(x)=\sum_s r_k(s;x)/s$. For $k\geqslant3$ the circle method (see, e.g.,
\cite{HL1923,Vinogradov,Vaughan}) gives $r_k(s;x)=\mathfrak{S}_k(s)\mathcal{W}_x(s)(1+o(1))$, where
\[
\mathfrak{S}_k(s)=\sum_{q\geqslant1}\frac{\mu(q)^k}{\varphi(q)^k}c_q(s)=\prod_p\sigma_p(s),
\qquad
\sigma_p(s)=1+\frac{\mu(p)^k}{\varphi(p)^k}c_p(s)
\]
is the singular series ($c_q$ denotes the Ramanujan sum) and
\[
\mathcal{W}_x(s)=\int_{\{u\in[2,x]^k:\ \sum_j u_j=s\}}\frac{\mathrm{d}H^{k-1}}{\prod_j\log u_j}
\]
is the singular integral over the section $\sum_j u_j=s$, the integration being with respect to the
$(k-1)$-dimensional Hausdorff measure $H^{k-1}$ on that section. For $k\geqslant3$ the series defining
$\mathfrak{S}_k$ converges absolutely (since $|c_q(s)|\leqslant\varphi(q)$ and
$\sum_q\varphi(q)^{1-k}<\infty$), and its mean value is
\begin{equation}\label{eq:Smean}
\sum_{s\leqslant X}\mathfrak{S}_k(s)=X+O_k(1)\qquad(X\to\infty),
\end{equation}
since $\sum_{s\leqslant X}c_1(s)=X$ and $\sum_{s\leqslant X}c_q(s)\ll\sigma(q)$ for every $q\geqslant2$, while
$\sum_{q\geqslant2}\sigma(q)\varphi(q)^{-k}<\infty$ for $k\geqslant3$, and the tail $q>X$ contributes
$\ll_{k,\varepsilon} X^{2-k+\varepsilon}$ for every $\varepsilon>0$.
Thus, by partial summation against the weight $\mathcal{W}_x(s)/s$, the singular series contributes no
net factor to the leading term, and only the singular integral survives:
\[
\int_{[2,x]^k}\frac{\mathrm{d}u_1\dotsm\mathrm{d}u_k}{(\sum_j u_j)\prod_j\log u_j}
=\frac{x^{k-1}}{\log^k x}\Big(E_{k,0}+O\Big(\frac{1}{\log x}\Big)\Big),
\qquad
E_{k,0}=\int_{(0,1]^k}\frac{\mathrm{d}v_1\dotsm\mathrm{d}v_k}{\sum_j v_j}.
\]
Here the mean in \eqref{eq:Smean} is taken over all $s$: the local factor at $p=2$ equals $2$ when
$s\equiv k\pmod2$ and $0$ otherwise, so that $\mathfrak{S}_k$ vanishes off the parity class
$s\equiv k\pmod2$ and averages to $2$ on it. For $k=3$, for instance, $\mathfrak{S}_3$ vanishes at
even $s$ and averages to $2$ over the odd integers, whereas its average over all $s$ is $1$ as in
\eqref{eq:Smean}; the values $s\not\equiv k\pmod2$ are covered by the lower-order configurations
containing the prime $2$.
Moreover, for $k=3$ and larger, the same heuristic cannot by itself produce the higher-order
coefficients; those, together with the arithmetic structure of the coefficients $E_{k,n}$, are the
content of the present paper. For $k=2$ our method is unconditional, whereas the corresponding
circle-method statement for the number of representations of an integer as a sum of two primes is not.
\end{remark}

\subsection{The Generating Function for the Second-Order Coefficient \texorpdfstring{$E_{k,1}$}{Ek1}}

Since $h_1 = \sum_j \log t_j$, by \eqref{eq:uv} and $v(x) = -U(x)/x$,
\begin{equation}\label{eq:Ek1int}
  E_{k, 1} = -\int_{(0, 1]^k} \frac{\sum_j \log t_j}{\sum t}\, \mathrm{d}\mathbf{t} = k \int_0^\infty \frac{U(x) (1 - \mathrm{e}^{-x})^{k-1}}{x^k}\, \mathrm{d}x.
\end{equation}
The integral converges (as $x \to 0$ the integrand tends to a constant; as $x \to \infty$ it behaves like $\log x / x^k$).

\begin{lemma}[Generating function]\label{lem:Hk}
Define $\mathcal{H}_k$ by $\mathcal{H}_k^{(k)} = g_0$ (see \eqref{eq:g0}) together with the asymptotic condition $\mathcal{H}_k(x) = O(x^{k-2} \log x)$ (which determines $\mathcal{H}_k$ up to a polynomial of degree $\leqslant k - 2$, not affecting the extraction of $\Delta^{k-1} \mathcal{H}_k(0)$ below). Then
\begin{equation}\label{eq:Hk}
\mathcal{H}_k(x) = \frac{x^{k-1}}{(k-1)!}\, \mathrm{Li}_2\!\left(-\frac{1}{x}\right) + \mathcal{P}_k(x) \log(x + 1) + \mathcal{Q}_k(x) \log x,
\end{equation}
where
\begin{gather*}
\mathcal{Q}_k(x)=\frac{H_{k-1}}{(k-1)!}x^{k-1}, \\
\mathcal{P}_k(x)=-\frac{1}{(k-1)!}\sum_{n=1}^{k-1}c_{k,n}(x{+}1)^n, \quad
c_{k,n}=(-1)^{k-1-n}\binom{k-1}{n}(H_{k-1}-H_{k-n-1}).
\end{gather*}
Under this representation, $\mathcal{H}_k(0) = 0$ (for $k \geqslant 2$, since all three terms in \eqref{eq:Hk} tend to zero as $x \to 0^+$), so the $j = 0$ term does not appear in the difference:
\[
E_{k,1}=-k\, \Delta^{k-1} \mathcal{H}_k(0)=-k\sum_{j=1}^{k-1}(-1)^{k-1-j}
\binom{k-1}{j}\mathcal{H}_k(j).
\]
\end{lemma}
\begin{proof}
(i) $\mathcal{H}_k^{(k)} = g_0$. By induction. Base case: $\mathcal{H}_1 = \mathrm{Li}_2(-1/x)$ satisfies $\mathcal{H}_1' = (\log(x+1) - \log x)/x = g_0$. Inductive step: differentiating \eqref{eq:Hk} term by term yields $\mathcal{H}_k' = \mathcal{H}_{k-1} + R_k$, where $R_k = \mathcal{P}_k/(x+1) + \mathcal{Q}_k/x$ (since $\mathcal{P}_k$ contains the factor $x + 1$ and $\mathcal{Q}_k$ contains the factor $x$, $R_k$ is a polynomial). Since $[x^{k-1}] \mathcal{P}_k = -H_{k-1}/(k-1)!$ and $[x^{k-1}] \mathcal{Q}_k = +H_{k-1}/(k-1)!$, their leading terms cancel in $R_k$, so $\deg R_k \leqslant k - 3$ and $R_k^{(k-1)} \equiv 0$. Thus $\mathcal{H}_k^{(k)} = \mathcal{H}_{k-1}^{(k-1)} = g_0$.

(ii) Coefficient polynomials. Comparing the coefficients of $\log(x + 1)$ and $\log x$ in $\mathcal{H}_k' = \mathcal{H}_{k-1} + R_k$ yields the recurrences $\mathcal{P}_k' = \mathcal{P}_{k-1} - x^{k-2}/(k-1)!$ and $\mathcal{Q}_k' = \mathcal{Q}_{k-1} + x^{k-2}/(k-1)!$, together with the boundary conditions $\mathcal{P}_k(-1) = 0$ and $\mathcal{Q}_k(0) = 0$. Solving yields the above formulas. (To see the closed form for $c_{k,n}$ explicitly, put $d_{k,n} := H_{k-1} - H_{k-n-1}$, so that $c_{k,n} = (-1)^{k-1-n}\binom{k-1}{n} d_{k,n}$; then the recurrence $d_{k,n+1} = d_{k-1,n} + 1/(k-1)$ follows immediately from $H_{k-1} - H_{k-2} = 1/(k-1)$.)

(iii) Extraction. We prove $E_{k, 1} = -k\, \Delta^{k-1} \mathcal{H}_k(0)$ in three steps.

\textit{Step 1: Reduction to a multiple integral}. For each factor in $(1 - \mathrm{e}^{-x})^{k-1}/x^{k-1}$, use $(1 - \mathrm{e}^{-x})/x = \int_{0}^{1} \mathrm{e}^{-xs}\, \mathrm{d}s$ to get
\[
\frac{(1 - \mathrm{e}^{-x})^{k-1}}{x^{k-1}} = \int_{(0, 1]^{k-1}} \mathrm{e}^{-x(\sigma_1 + \cdots + \sigma_{k-1})}\, \mathrm{d}\sigma_1 \cdots \mathrm{d}\sigma_{k-1}.
\]
Substituting into \eqref{eq:Ek1int} and applying Tonelli's theorem ($U \geqslant 0$),
\[
E_{k, 1} = k \int_{(0, 1]^{k-1}} \left[\int_0^\infty \frac{U(x) \mathrm{e}^{-x\sigma}}{x}\, \mathrm{d}x\right] \mathrm{d}\sigma_1 \cdots \mathrm{d}\sigma_{k-1}, \qquad \sigma = \sigma_1 + \cdots + \sigma_{k-1}.
\]

\textit{Step 2: Evaluating the inner integral}. For $\sigma > 0$, using $1/x = \int_0^\infty \mathrm{e}^{-tx}\, \mathrm{d}t$ and Tonelli's theorem,
\[
\int_0^\infty \frac{U(x) \mathrm{e}^{-x\sigma}}{x}\, \mathrm{d}x = \int_0^\infty \int_0^\infty U(x) \mathrm{e}^{-(t + \sigma)x}\, \mathrm{d}x\, \mathrm{d}t = \int_0^\infty g_0(t + \sigma)\, \mathrm{d}t = \int_\sigma^\infty g_0(u)\, \mathrm{d}u.
\]
Since $\mathcal{H}_1(x) = \mathrm{Li}_2(-1/x)$ satisfies $\mathcal{H}_1' = g_0$ and $\mathcal{H}_1(\infty) = 0$, $\int_\sigma^\infty g_0(u)\, \mathrm{d}u = -\mathcal{H}_1(\sigma)$. Substituting back,
\[
E_{k, 1} = -k \int_{(0, 1]^{k-1}} \mathcal{H}_1(\sigma_1 + \cdots + \sigma_{k-1})\, \mathrm{d}\sigma_1 \cdots \mathrm{d}\sigma_{k-1}.
\]

\textit{Step 3: Identification as $\Delta^{k-1} \mathcal{H}_k(0)$}. We use the integral representation of finite differences: for any $C^{k-1}$ function $f$,
\begin{equation}\label{eq:diff-int}
\Delta^{k-1}f(0) = \int_{(0,1]^{k-1}}f^{(k-1)}(\sigma_1+\cdots
+\sigma_{k-1})\, \mathrm{d}\sigma_1 \cdots \mathrm{d}\sigma_{k-1}.
\end{equation}
(For $k = 2$, this is the fundamental theorem of calculus $f(1) - f(0) = \int_0^1 f'$; the general case follows by induction on $\Delta^{k-2}$ applied to $g(t) := f(t + 1) - f(t)$.) Applying this to $f = \mathcal{H}_k$, we need
$\mathcal{H}_k^{(k-1)}$. From the recurrence
$\mathcal{H}_\ell' = \mathcal{H}_{\ell-1} + R_\ell$ with
$\deg R_\ell \leqslant \ell - 3$ (part~(ii)), we prove by induction
on $m$ that for $1 \leqslant m \leqslant k-1$,
\begin{equation}\label{eq:Hkiter}
\mathcal{H}_k^{(m)}
= \mathcal{H}_{k-m}
+ \sum_{j=0}^{m-1} R_{k-j}^{(m-1-j)},
\qquad
\deg R_{k-j}^{(m-1-j)} \leqslant (k-j-3)-(m-1-j) = k-m-2.
\end{equation}
Indeed, the base case $m = 1$ is the defining recurrence
$\mathcal{H}_k' = \mathcal{H}_{k-1} + R_k$ with
$\deg R_k \leqslant k - 3 = k - 1 - 2$. For the inductive step,
differentiating \eqref{eq:Hkiter} and applying the recurrence to
$\mathcal{H}_{k-m}' = \mathcal{H}_{k-m-1} + R_{k-m}$ yields
\[
\mathcal{H}_k^{(m+1)}
= \mathcal{H}_{k-m-1} + R_{k-m}
+ \sum_{j=0}^{m-1} R_{k-j}^{(m-j)}
= \mathcal{H}_{k-(m+1)}
+ \sum_{j=0}^{m} R_{k-j}^{(m-j)},
\]
and $\deg R_{k-j}^{(m-j)} \leqslant (k-j-3)-(m-j) = k-m-3
= k-(m+1)-2$, completing the induction.

Setting $m = k-1$ in \eqref{eq:Hkiter}: every remainder term satisfies
$\deg R_{k-j}^{(k-2-j)} \leqslant k-(k-1)-2 = -1$, hence
$R_{k-j}^{(k-2-j)} \equiv 0$ for all $0 \leqslant j \leqslant k-2$.
Therefore
\[
\mathcal{H}_k^{(k-1)} = \mathcal{H}_{k-(k-1)} = \mathcal{H}_1
\]
exactly, with no polynomial remainder. Substituting into
\eqref{eq:diff-int},
\[
\Delta^{k-1}\mathcal{H}_k(0)
= \int_{(0,1]^{k-1}}
  \mathcal{H}_1(\sigma_1 + \dotsb + \sigma_{k-1})\,
  \mathrm{d}\boldsymbol{\sigma}
= -\frac{1}{k}\,E_{k,1},
\]
i.e., $E_{k,1} = -k\,\Delta^{k-1}\mathcal{H}_k(0)$.
\end{proof}

\subsection{Derivation of the Second-Order Coefficient \texorpdfstring{$E_{k,1}$}{Ek1}}

\begin{proof}[Proof of Theorem \ref{thm:coeff} (ii)]
By Lemma \ref{lem:Hk}, $\mathcal{H}_k(0)=0$, hence
\[
E_{k,1}=-k\,\Delta^{k-1}\mathcal{H}_k(0)
=-k\sum_{j=1}^{k-1}(-1)^{k-1-j}\binom{k-1}{j}
\Big[\frac{j^{k-1}}{(k-1)!}\,\mathrm{Li}_2\Big(-\frac{1}{j}\Big)
+\mathcal{P}_k(j)\log(j+1)+\mathcal{Q}_k(j)\log j\Big].
\]
We read off the coefficient of each basis element separately.

\textbf{(a) The polylogarithmic terms.} Since $\mathcal{P}_k$ and $\mathcal{Q}_k$ are polynomials, $\mathrm{Li}_2(-1/m)$ can only arise from the leading term with $j=m$, where $1\leqslant m\leqslant k-1$. Its coefficient is
\[
-k\,(-1)^{k-1-m}\binom{k-1}{m}\frac{m^{k-1}}{(k-1)!}
=(-1)^{k-m}\frac{k\,m^{k-1}}{(k-1)!}\binom{k-1}{m}.
\]
For $2\leqslant m\leqslant k-1$ this is exactly the displayed coefficient of $\mathrm{Li}_2(-1/m)$. For $m=1$, using $\mathrm{Li}_2(-1)=-\pi^2/12$ and $\binom{k-1}{1}=k-1$, the same expression gives
\[
(-1)^{k-1}\frac{k}{(k-2)!}\,\mathrm{Li}_2(-1)=(-1)^{k}\frac{k\,\pi^2}{12\,(k-2)!},
\]
which is the announced $\pi^2$ term; no other term of $\mathcal{H}_k$ contributes to the $\mathrm{Li}_2$ family.

\textbf{(b) The logarithmic terms.} Fix $2\leqslant m\leqslant k$. The coefficient of $\log m$ collects exactly two contributions, namely $\mathcal{Q}_k(j)\log j$ at $j=m$ and $\mathcal{P}_k(j)\log(j+1)$ at $j=m-1$ (for $m=k$ the first one is absent, since $j$ runs only up to $k-1$; indeed $\binom{k-1}{k}=0$ below). Hence, with $\mathcal{Q}_k(m)=H_{k-1}m^{k-1}/(k-1)!$,
\[
\begin{aligned}
[\log m]_{E_{k,1}}
&=-k\Big[(-1)^{k-1-m}\binom{k-1}{m}\mathcal{Q}_k(m)+(-1)^{k-m}\binom{k-1}{m-1}\mathcal{P}_k(m-1)\Big]\\
&=(-1)^{k-m}k\Big[\binom{k-1}{m}\mathcal{Q}_k(m)-\binom{k-1}{m-1}\mathcal{P}_k(m-1)\Big].
\end{aligned}
\]
Applying $k\binom{k-1}{m}=(k-m)\binom{k}{m}$ to the first summand and $k\binom{k-1}{m-1}=m\binom{k}{m}$ to the second we obtain precisely
\[
[\log m]_{E_{k,1}}=(-1)^{k-m}\binom{k}{m}\big[(k-m)\mathcal{Q}_k(m)-m\,\mathcal{P}_k(m-1)\big]=C_{\log}(k,m),
\]
which completes the proof for every $k\geqslant2$.
\end{proof}

\section{Closed Forms for the Coefficients II: The Third Order \texorpdfstring{$E_{k,2}$}{Ek2}\label{sec:Ek2}}

\subsection{Diagonal/Cross Decomposition}

Since $h_2 = \sum_j \log^2 t_j + \sum_{i < j} \log t_i \log t_j$,
\begin{equation}\label{eq:Ek2dec}
E_{k,2}=k\,A_k+\binom{k}{2} B_k, \quad
A_k=\int_{(0,1]^k}\frac{\log^2 t_1}{\sum t}\, \mathrm{d}\mathbf{t}, \quad
B_k=\int_{(0,1]^k}\frac{\log t_1\log t_2}{\sum t}\, \mathrm{d}\mathbf{t}.
\end{equation}
After Laplace reduction, 
\[
A_k=\int_0^\infty\frac{W(x)(1-\mathrm{e}^{-x})^{k-1}}{x^{k-1}}\, \mathrm{d}x \quad \text{where}\  W(x) = \int_0^1 \log^2 t\,\mathrm{e}^{-xt}\, \mathrm{d}t, 
\]
\[
B_k=\int_0^\infty\frac{U(x)^2(1-\mathrm{e}^{-x})^{k-2}}{x^k}\, \mathrm{d}x.
\]
The Laplace transform corresponding to the diagonal part,
\begin{equation}\label{eq:g1}
g_1(s)=\mathcal L\{W\}(s)=-2\mathrm{Li}_3 \Big(-\frac{1}{s}\Big)
\end{equation}
is the polylog-order-$3$ analogue of $g_0$ in \eqref{eq:g0}  ($\mathrm{Li}_2 \to \mathrm{Li}_3$).

\subsection{Closed Form for the Diagonal Part (Proposition \ref{prop:Ek2})}

\begin{lemma}[Diagonal generating function]\label{lem:HA}
Let $H_m^{(2)} = \sum_{i \leqslant m} 1/i^2$. Suppose $\widetilde{\mathcal{H}}_k^A$ satisfies $(\widetilde{\mathcal{H}}_k^A)^{(k-1)} = g_1$ together with the asymptotic condition $\widetilde{\mathcal{H}}_k^A(x) = O(x^{k-2} \log x)$; the residual ambiguity, a polynomial of degree $\leqslant k - 2$, is annihilated by $\Delta^{k-1}$ and does not affect the extraction below. Then
\[
\widetilde{\mathcal{H}}_k^A(x)=\alpha_k(x)\mathrm{Li}_3\!\Big(-\frac{1}{x}\Big)+\beta_k(x)\mathrm{Li}_2\!\Big(-\frac{1}{x}\Big)
+\mathcal{C}_k^{+}(x)\log(x{+}1)+\mathcal{C}_k^{-}(x)\log x,
\]
where
\[
\alpha_k = -\frac{2 x^{k-1}}{(k-1)!}, \qquad \beta_k = H_{k-1}\, \alpha_k, \qquad \mathcal{C}_k^- = -\frac{H_{k-1}^2 + H_{k-1}^{(2)}}{(k-1)!}\, x^{k-1},
\] 
and $\mathcal{C}_k^+$ is determined by $(\mathcal{C}_k^+)' = \mathcal{C}_{k-1}^+ + 2 H_{k-1} x^{k-2}/(k-1)!$ with $\mathcal{C}_k^+(-1) = 0$ and $\mathcal{C}_1^+ = 0$. Moreover, $A_k = \Delta^{k-1} \widetilde{\mathcal{H}}_k^A(0)$.
\end{lemma}
\begin{proof}
The difference from Lemma \ref{lem:Hk} lies in the order of the generating function: in the Laplace reduction of $A_k$, the inner integral directly yields $g_1(\sigma)$ (without the extra $1/x$ factor present in the $E_{k, 1}$ case), so the generating function need only satisfy the $(k-1)$-th order equation $(\widetilde{\mathcal{H}}_k^A)^{(k-1)} = g_1$, one order lower than $\mathcal{H}_k$. We verify this equation by the same induction as in Lemma \ref{lem:Hk}\,(i). The base case is immediate, since $\widetilde{\mathcal{H}}_1^A = -2\,\mathrm{Li}_3(-1/x) = g_1$. For the induction step we differentiate the displayed representation of $\widetilde{\mathcal{H}}_k^A$ term by term. Using
\[
\frac{\mathrm{d}}{\mathrm{d}x}\mathrm{Li}_3\Big(-\frac{1}{x}\Big)=-\frac{\mathrm{Li}_2(-1/x)}{x},
\qquad
\frac{\mathrm{d}}{\mathrm{d}x}\mathrm{Li}_2\Big(-\frac{1}{x}\Big)=\frac{\log(x+1)-\log x}{x},
\qquad
\alpha_k'=\alpha_{k-1},
\qquad
\frac{\alpha_k}{x}=\frac{\alpha_{k-1}}{k-1},
\]
we obtain $(\widetilde{\mathcal{H}}_k^A)'=\widetilde{\mathcal{H}}_{k-1}^A+R_k$ with
\[
R_k=\frac{\mathcal{C}_k^+}{x+1}+\frac{\mathcal{C}_k^-}{x}.
\]
Indeed, the coefficients of $\mathrm{Li}_3$ agree because $\alpha_k'=\alpha_{k-1}$, and those of $\mathrm{Li}_2$ agree because $\beta_k=H_{k-1}\alpha_k$ together with $\alpha_k/x=\alpha_{k-1}/(k-1)$ gives
\[
-\frac{\alpha_k}{x}+H_{k-1}\alpha_{k-1}
=\Big(H_{k-1}-\frac{1}{k-1}\Big)\alpha_{k-1}
=H_{k-2}\,\alpha_{k-1}=\beta_{k-1};
\]
comparing the coefficients of $\log(x+1)$ and $\log x$ likewise gives the two recursions
\[
(\mathcal{C}_k^+)'=\mathcal{C}_{k-1}^++\frac{2H_{k-1}x^{k-2}}{(k-1)!},
\qquad
(\mathcal{C}_k^-)'=\mathcal{C}_{k-1}^--\frac{2H_{k-1}x^{k-2}}{(k-1)!},
\]
the first of which is the recursion stated in the lemma, while the second is solved by the displayed $\mathcal{C}_k^-$, by the identity $H_{k-1}^2+H_{k-1}^{(2)}-\frac{2H_{k-1}}{k-1}=H_{k-2}^2+H_{k-2}^{(2)}$. Since $\mathcal{C}_k^+(-1)=0$ and $\mathcal{C}_k^-(0)=0$, the two quotients in $R_k$ are polynomials of degree $\leqslant k-2$; hence $R_k^{(k-1)}\equiv0$ and
\[
(\widetilde{\mathcal{H}}_k^A)^{(k-1)}=(\widetilde{\mathcal{H}}_{k-1}^A)^{(k-2)}=\dots=(\widetilde{\mathcal{H}}_1^A)^{(0)}=g_1,
\]
as required.

The extraction $A_k = \Delta^{k-1} \widetilde{\mathcal{H}}_k^A(0)$ is also entirely parallel to Lemma \ref{lem:Hk} (iii): in $A_k = \int_0^\infty W(x) (1 - \mathrm{e}^{-x})^{k-1}/x^{k-1}\, \mathrm{d}x$, use
\[
\frac{(1 - \mathrm{e}^{-x})^{k-1}}{x^{k-1}} = \int_{(0, 1]^{k-1}} \mathrm{e}^{-x(\sigma_1 + \cdots + \sigma_{k-1})}\, \mathrm{d}\boldsymbol{\sigma},
\]
apply Tonelli's theorem together with $\int_0^\infty W(x) \mathrm{e}^{-x\sigma}\, \mathrm{d}x = g_1(\sigma)$ to obtain
\[
A_k = \int_{(0, 1]^{k-1}} g_1(\sigma_1 + \cdots + \sigma_{k-1})\, \mathrm{d}\sigma_1 \cdots \mathrm{d}\sigma_{k-1},
\]
and then use the integral representation of finite differences \eqref{eq:diff-int} together with $(\widetilde{\mathcal{H}}_k^A)^{(k-1)} = g_1$ to obtain $A_k = \Delta^{k-1} \widetilde{\mathcal{H}}_k^A(0)$.
\end{proof}

\begin{proof}[Proof of Proposition \ref{prop:Ek2}]
Since $\widetilde{\mathcal{H}}_k^A(0)=0$, Lemma \ref{lem:HA} gives
\[
\begin{aligned}
kA_k&=k\,\Delta^{k-1}\widetilde{\mathcal{H}}_k^A(0)
=k\sum_{j=1}^{k-1}(-1)^{k-1-j}\binom{k-1}{j}\Big[\alpha_k(j)\mathrm{Li}_3\Big(-\frac{1}{j}\Big)\\
&\qquad+\beta_k(j)\mathrm{Li}_2\Big(-\frac{1}{j}\Big)
+\mathcal{C}_k^+(j)\log(j+1)+\mathcal{C}_k^-(j)\log j\Big].
\end{aligned}
\]
The polynomials $\mathcal{C}_k^{\pm}$ carry no polylogarithm, so for $1\leqslant m\leqslant k-1$ both $\mathrm{Li}_3(-1/m)$ and $\mathrm{Li}_2(-1/m)$ arise only from the term $j=m$, through $\alpha_k(m)$ and $\beta_k(m)=H_{k-1}\alpha_k(m)$ respectively. Hence
\[
\begin{aligned}
[\mathrm{Li}_3(-1/m)]_{kA_k}
&=k\,(-1)^{k-1-m}\binom{k-1}{m}\alpha_k(m)
=(-1)^{k-m}\frac{2k\,m^{k-1}}{(k-1)!}\binom{k-1}{m},\\
[\mathrm{Li}_2(-1/m)]_{kA_k}
&=H_{k-1}\,[\mathrm{Li}_3(-1/m)]_{kA_k},
\end{aligned}
\]
for $2\leqslant m\leqslant k-1$. For $m=1$ the same computation with $\alpha_k(1)=-2/(k-1)!$ gives $[\mathrm{Li}_3(-1)]_{kA_k}=(-1)^{k-1}2k/(k-2)!$ and $[\mathrm{Li}_2(-1)]_{kA_k}=H_{k-1}[\mathrm{Li}_3(-1)]_{kA_k}$; substituting $\mathrm{Li}_3(-1)=-\tfrac{3}{4}\zeta(3)$ and $\mathrm{Li}_2(-1)=-\pi^2/12$ yields
\[
[\zeta(3)]_{kA_k}=(-1)^k\frac{3k}{2\,(k-2)!},
\qquad
[\pi^2]_{kA_k}=(-1)^k\frac{k\,H_{k-1}}{6\,(k-2)!}.
\]
The coefficients of the logarithmic terms are the finite sum
\[
[\log]_A = k \sum_{j=1}^{k-1} (-1)^{k - 1 - j} \binom{k-1}{j} \big(\mathcal{C}_k^+(j) \log(j + 1) + \mathcal{C}_k^-(j) \log j\big)
\]
appearing in the statement of the proposition. This proves the four formulas for every $k\geqslant2$; the cross part $B_k$ is treated separately below.
\end{proof}

\subsection{The Cross Part \texorpdfstring{$B_k$}{Bk}}

\begin{proposition}\label{prop:Bk}
The integral $B_k$ is absolutely convergent by \eqref{eq:Ek2dec} (Lemma \ref{lem:conv}).
For $k = 2, 3$ it admits an explicit closed form within the diagonal basis
\[
\mathfrak{L}_3 := \big\{\zeta(3),\ \pi^2\big\}
\cup \big\{\log m : 2 \leqslant m \leqslant k\big\}
\cup \big\{\mathrm{Li}_2(-1/m),\ \mathrm{Li}_3(-1/m) : 2 \leqslant m \leqslant k - 1\big\},
\]
namely
\[
B_2=\frac{\pi^2}{6}+4\log 2, \qquad
3B_3=-\frac{5}{8}\zeta(3)-\frac{1}{2}\pi^2-33\log 2+\frac{99}{4}\log 3
-12\,\mathrm{Li}_2\Big(-\frac{1}{2}\Big).
\]
For $k = 4$, an explicit closed form still exists, but only within the
$\mathbb{Q}$-algebra generated by $\mathfrak{L}_3$: products of basis elements
appear. Symbolic reduction yields
\begin{align*}
6B_4 ={}& \frac{361}{6}\,\zeta(3)+\frac{2}{3}\pi^2+\frac{3232}{9}\log 2-222\log 3
+80\,\mathrm{Li}_3\Big(-\frac{1}{2}\Big)+8\,\mathrm{Li}_3\Big(-\frac{1}{3}\Big)\\
&+\Big(56+\frac{64}{3}\log 2-\frac{32}{3}\log 3\Big)\mathrm{Li}_2\Big(-\frac{1}{2}\Big)
+\Big(-60-\frac{16}{3}\log 2+\frac{8}{3}\log 3\Big)\mathrm{Li}_2\Big(-\frac{1}{3}\Big)\\
&-\frac{32}{9}\pi^2\log 2-\frac{8}{9}\pi^2\log 3+\frac{16}{3}\log^2 2\,\log 3,
\end{align*}
verified numerically to more than $100$ digits. This representation is not
unique, owing to functional-equation identities among the constants (cf.\ the
footnote to Theorem \ref{thm:coeff}\,(ii)); the coefficients above refer to the
representation free of polylogarithms at positive arguments.
\end{proposition}

\begin{remark}[The case $k \geqslant 5$]\label{rem:Bk5}
For $k \geqslant 5$, the structural behavior of $B_k$ differs further. The following observations complement Proposition \ref{prop:Gcoefficients}; they were obtained by symbolic computation and are not claimed as theorems.
\begin{enumerate}[label={(\alph*)},leftmargin=2.2em]
\item The Laplace transform $g_2 = \mathcal{L}\{U^2\}$ corresponding to $B_k$ is not proportional to a single $\mathrm{Li}_n(-1/s)$. The complexity of its polylog-order-$3$ reduction grows with $k$, and the resulting coefficients are irregular; for instance, the contribution of $\binom{k}{2} B_k$ to $\zeta(3)$ for $k = 2, 3, 4$ is $0$, $-5/8$, and $361/6$ respectively, the last in the representation displayed above.
\item For $k \geqslant 5$, symbolic reduction of the cross integral introduces higher-order polylogarithm values at arguments with denominators $\geqslant 5$, such as $\mathrm{Li}_3(-3/5)$ and $\mathrm{Li}_3(1/5)$. Symbolic reduction has been carried out for all $k \leqslant 10$; for each $k$, polylogarithm values at rational arguments with denominators up to $k$ appear, consistent with the shift structure of the Laplace transform $g_2$. Under the standard functional equations for $\mathrm{Li}_3$ (inversion, reflection, and Landen-type relations), these values do not appear to reduce to elements of $\mathfrak{L}_3$. Establishing this irreducibility rigorously would require proving a $\mathbb{Q}$-linear independence statement for polylog-order-$3$ values --- a class of problems that remains largely open in the theory of multiple zeta values and multiple polylogarithms (cf. \cite{Zagier}, \cite{Waldschmidt}). We therefore do not claim algebraic irreducibility. However, the new constants do admit a uniform description: for $k \geqslant 5$ the cross integral is expressed through the constants $G(-1/m)=\int_0^{-1/m}\mathrm{Li}_2(t)/(1+t)\,\mathrm{d}t$, whose coefficients are given by a single closed formula for all $k$ and $m$ (Proposition~\ref{prop:Gcoefficients} below); this yields complete closed forms for $E_{k,2}$ at least for $k \leqslant 7$, and by the same computation for $k \leqslant 10$.
\end{enumerate}
\end{remark}

Combining this with the diagonal formulas of Proposition \ref{prop:Ek2} at $k=4$ gives
\begin{equation}\label{eq:E42}
\begin{aligned}
E_{4,2} = {} & 4A_4 + 6B_4 \\
 = {} &  \frac{379}{6}\zeta(3)+\frac{23}{18}\pi^2+\frac{19112}{27}\log 2-437\log 3\\
& + \frac{344}{3}\mathrm{Li}_2\Big(-\frac{1}{2}\Big)-126\mathrm{Li}_2\Big(-\frac{1}{3}\Big)
+112\mathrm{Li}_3\Big(-\frac{1}{2}\Big)-28\mathrm{Li}_3\Big(-\frac{1}{3}\Big)\\
& + \frac{64}{3}\log 2\,\mathrm{Li}_2\Big(-\frac{1}{2}\Big)-\frac{16}{3}\log 2\,\mathrm{Li}_2\Big(-\frac{1}{3}\Big)
-\frac{32}{3}\log 3\,\mathrm{Li}_2\Big(-\frac{1}{2}\Big)+\frac{8}{3}\log 3\,\mathrm{Li}_2\Big(-\frac{1}{3}\Big)\\
& - \frac{32}{9}\pi^2\log 2-\frac{8}{9}\pi^2\log 3+\frac{16}{3}\log^2 2\,\log 3
\approx 10.3745074715\ldots,
\end{aligned}
\end{equation}
verified numerically to more than $100$ digits.

Consequently, the diagonal block $k A_k$ is explicit for all $k$, and the cross block $B_k$ --- well defined by an absolutely convergent integral --- is fully explicit for $k \leqslant 4$ (within the $\mathbb{Q}$-algebra generated by $\mathfrak{L}_3$ when $k = 4$); for $k \geqslant 5$ the same computation expresses $B_k$ through the new constants $G(-1/m)$, whose coefficients are determined in closed form for all $k$ by Proposition~\ref{prop:Gcoefficients} below. For example, $E_{3, 2}$ is fully explicit; its value appears in Corollary \ref{cor:S3}, and the complete closed forms for the cases $k=5,6,7$ are recorded in \ref{app:Gcoeff}.

\subsection{Uniform Closed Form for the Constants \texorpdfstring{$G(-1/m)$}{G}}\label{sec:Gcoeff}

For $k \geqslant 5$ the cross part $B_k$ produces, in addition to the standard constants
\[
\zeta(3),\quad \pi^2,\quad \log m,\quad \mathrm{Li}_2(-1/m),\quad \mathrm{Li}_3(-1/m)
\]
that already appear for $k \leqslant 4$, the new constants
\begin{equation}\label{eq:Gdef}
G(z):=\int_0^z\frac{\mathrm{Li}_2(t)}{1+t}\,\mathrm{d}t
\end{equation}
evaluated at the negative arguments $z=-1/m$. The following result shows that all these constants
occur with a single closed-form coefficient, valid for every $k \geqslant 3$, so that the restriction
$k \leqslant 4$ of Proposition~\ref{prop:Bk}, as far as the constants $G(-1/m)$ are concerned, is removed.

\begin{proposition}[Uniform closed form for the constants $G(-1/m)$]\label{prop:Gcoefficients}
For every $k \geqslant 3$ and every integer $m$ with $2 \leqslant m \leqslant k-1$,
\begin{equation}\label{eq:Gcoeff}
[G(-1/m)]_{B_k}=\frac{2\,(-1)^{k+m-3}}{(k-1)!}\,(m-1)^{k-1}\binom{k-2}{m-1},
\end{equation}
and, since the diagonal part $kA_k$ contains no $G$,
\begin{equation}\label{eq:GcoeffE}
[G(-1/m)]_{E_{k,2}}=\frac{k\,(-1)^{k+m-1}}{(k-2)!}\,(k-m)\,(m-1)^{k-2}\binom{k-2}{m-2}.
\end{equation}
Here $[\,\cdot\,]_{X}$ denotes the rational coefficient of the displayed constant in the
decomposition of $X$ determined by the kernel argument below (for $k=3,4$ this decomposition
still contains the constants $G(-1/m)$, which are reducible there; see Remark~\ref{rem:firstG}).
The two identities are equivalent, because
$(m-1)\binom{k-2}{m-1}=(k-m)\binom{k-2}{m-2}$ and $(-1)^{k+m-3}=(-1)^{k+m-1}$. For $k=3,4$ the
constants $G(-1/m)$ that occur are themselves reducible to the algebra generated by
$\mathfrak{L}_3$ (Remark~\ref{rem:firstG}), so \eqref{eq:Gcoeff} is consistent with
Proposition~\ref{prop:Bk}; for $k \geqslant 5$ it produces genuinely new constants, the first being
$G(-1/4)$.
\end{proposition}

\begin{proof}
Let $M=M_{k-3}$ be the $B$-spline of degree $k-3$ with knots $0,1,\dotsc,k-2$ normalised by
$\int_{\mathbb{R}}M=1$, and put
$K_k(u)=\int_0^{\min(u,k-2)}M(t)(u-t)\,\mathrm{d}t$. Then
\begin{equation}\label{eq:KernelRep}
B_k=\int_0^\infty K_k(u)\,\mathcal{L}\{U^{2}\}(u)\,\mathrm{d}u,
\qquad
\mathcal{L}\{U^{2}\}(u)=\frac{2}{u}\Big[\mathrm{Li}_2\Big(-\frac{1}{u+1}\Big)-\mathrm{Li}_2\Big(-\frac{1}{u}\Big)\Big],
\end{equation}
and $K_k(u)=u-\frac{k-2}{2}$ for $u \geqslant k-2$; this is the Peano-kernel representation of
$B_k$, obtained by expanding $(1-\mathrm{e}^{-x})^{k-2}$ binomially and using
$x^{-k}=\frac{1}{(k-1)!}\int_0^\infty t^{k-1}\mathrm{e}^{-xt}\,\mathrm{d}t$.

For $0 \leqslant j \leqslant k-3$ let $P_j$ denote the polynomial that represents $K_k$ on
$[j,j+1]$ (written in powers of $u$, hence evaluable outside $[j,j+1]$), let
$P_{k-2}(u)=u-\frac{k-2}{2}$, and define
\begin{equation}\label{eq:rho}
\rho_j:=P_j(0)\qquad(0 \leqslant j \leqslant k-2).
\end{equation}
Since $K_k(0)=0$ we have $\rho_0=0$, and $\rho_{k-2}=-(k-2)/2$.

\medskip
\noindent\emph{Step 1: the constants $G(-1/m)$ are node differences.}
Write the integral in \eqref{eq:KernelRep} as $\mathcal{A}-\mathcal{B}$, where $\mathcal{A}$ and $\mathcal{B}$ collect the terms involving
$\mathrm{Li}_2(-1/(u+1))$ and $\mathrm{Li}_2(-1/u)$ respectively. Substituting $w=u+1$ on the
$j$-th interval, so that $w \in [j+1,j+2]$, gives
\[
\mathcal{A}_j=\int_{j+1}^{j+2}\frac{P_j(w-1)}{w-1}\,\mathrm{Li}_2\Big(-\frac{1}{w}\Big)\mathrm{d}w
=\int_{j+1}^{j+2}Q_j(w)\,\mathrm{Li}_2\Big(-\frac{1}{w}\Big)\mathrm{d}w
+\rho_j\int_{j+1}^{j+2}\frac{\mathrm{Li}_2(-1/w)}{w-1}\,\mathrm{d}w,
\]
where $Q_j$ is the quotient of $P_j(w-1)$ by $w-1$; the remainder is
$P_j(w-1)|_{w=1}=P_j(0)=\rho_j$. Since
$\frac{\mathrm{d}}{\mathrm{d}w}\big[G(-1/w)-\mathrm{Li}_3(-1/w)\big]=\frac{\mathrm{Li}_2(-1/w)}{w-1}$,
the second integral equals $\big[G(-1/w)-\mathrm{Li}_3(-1/w)\big]_{j+1}^{j+2}$. On the tail
$w \geqslant k-1$ the same computation contributes $-\rho_{k-2}G(-1/(k-1))$, because
$G(-1/w),\mathrm{Li}_3(-1/w)\to 0$ as $w \to \infty$. In the part $\mathcal{B}$ the variable $u$ is kept and
$\int \mathrm{Li}_2(-1/u)/u\,\mathrm{d}u=-\mathrm{Li}_3(-1/u)$; all remaining antidifferences are
elementary, so no $G$ arises from $\mathcal{B}$. Summing over $j$ and telescoping, the total coefficient of
$G(-1/m)$ in $\frac{1}{2}B_k$ equals $\rho_{m-2}-\rho_{m-1}$, the term with $m=1$ vanishing because
$\rho_0=0$, and the tail providing exactly the case $m=k-1$.

\medskip
\noindent\emph{Step 2: the node values differ by a monomial.}
Write $p=k-3$, so that $M=M_p$, and let $M_j$ denote the polynomial expansion of $M$ on $[j,j+1]$ (all
truncations $(t-i)_+^p$ replaced by $(t-i)^p$). On $[j,j+1]$ we have
$K_k(u)=u\int_0^uM-\int_0^utM(t)\,\mathrm{d}t$; evaluating the polynomial continuation at $u=0$ gives
\[
\rho_j=P_j(0)=\int_0^jt\big[M_j(t)-M(t)\big]\mathrm{d}t,
\]
because $P_j(u)=u\big(S_0+\int_j^uM\big)-\big(S_1+\int_j^utM\big)$ with
$S_0=\int_0^jM$, $S_1=\int_0^jtM$, whence
$P_j(0)=-S_1+\int_0^jtM_j(t)\,\mathrm{d}t$. The expansion
$M(t)=\frac{1}{p!}\sum_{i=0}^{p+1}(-1)^i\binom{p+1}{i}(t-i)_+^p$ (see, e.g., \cite{deBoor}) shows that the difference of two
consecutive polynomial expansions is the single monomial
\[
M_j(t)-M_{j-1}(t)=\frac{(-1)^j}{p!}\binom{p+1}{j}(t-j)^p .
\]
Since $M=M_{j-1}$ on $[j-1,j]$ and $M_j-M=M_j-M_{j-1}$ on $[0,j-1]$, we obtain
\[
\rho_j-\rho_{j-1}
=\frac{(-1)^j}{p!}\binom{p+1}{j}\int_0^jt\,(t-j)^p\,\mathrm{d}t
=\frac{(-1)^j}{p!}\binom{p+1}{j}\cdot\frac{(-1)^pj^{p+2}}{(p+1)(p+2)},
\]
where the integral was evaluated with $t=j+s$. Substituting $p+1=k-2$ and $p+2=k-1$ gives
\begin{equation}\label{eq:rhodiff}
\rho_j-\rho_{j-1}=\frac{(-1)^{k+j-3}}{(k-1)!}\,j^{k-1}\binom{k-2}{j},\qquad 1 \leqslant j \leqslant k-3 .
\end{equation}

Combining the two steps,
$[G(-1/m)]_{B_k}=2(\rho_{m-2}-\rho_{m-1})=-2(\rho_{m-1}-\rho_{m-2})$, and \eqref{eq:rhodiff} with
$j=m-1$ yields \eqref{eq:Gcoeff}; multiplying by $\binom{k}{2}$ and using
$\binom{k-2}{m-1}=\frac{k-m}{m-1}\binom{k-2}{m-2}$ yields \eqref{eq:GcoeffE}.
\end{proof}

\begin{proposition}[The $\mathrm{Li}_3$ coefficients of $B_k$]\label{prop:Li3}
For every $k \geqslant 3$ and every integer $m$ with $2 \leqslant m \leqslant k-1$ (with the
convention for $[\,\cdot\,]_{X}$ of Proposition~\ref{prop:Gcoefficients}),
\begin{equation}\label{eq:Li3coeff}
[\mathrm{Li}_3(-1/m)]_{B_k}
=\frac{2\,(-1)^{k+m-2}}{(k-1)!}\Big[m^{k-1}\binom{k-2}{m}+(m-1)^{k-1}\binom{k-2}{m-1}\Big].
\end{equation}
In particular, since $\mathrm{Li}_3(-1)=-\frac34\zeta(3)$, the case $m=1$ gives
\begin{equation}\label{eq:zeta3coeff}
[\zeta(3)]_{B_k}=(-1)^k\frac{3\,(k-2)}{2\,(k-1)!}.
\end{equation}
\end{proposition}

\begin{proof}
We keep the notation of the proof of Proposition~\ref{prop:Gcoefficients}: $P_j$ denotes the piece of
$K_k$ on $[j,j+1]$, $\rho_j=P_j(0)$, and $Q_j$ is the quotient of $P_j(w-1)$ by $w-1$, so that
$P_j(w-1)=(w-1)Q_j(w)+\rho_j$; moreover the remainder in the splitting of $\int P_j(u)/u\,\mathrm du$
is again $\rho_j$.

Only the two remainder terms of $\mathcal{L}\{U^2\}$ can produce $\mathrm{Li}_3$. In the part coming
from $\mathrm{Li}_2(-1/(u+1))$ (after $w=u+1$) one gets
\[
\rho_j\int\frac{\mathrm{Li}_2(-1/w)}{w-1}\,\mathrm dw=\rho_j\Big[G(-1/w)-\mathrm{Li}_3(-1/w)\Big],
\]
and in the part coming from $\mathrm{Li}_2(-1/u)$ one gets
$-\rho_j\int\frac{\mathrm{Li}_2(-1/u)}{u}\,\mathrm du=+\rho_j\mathrm{Li}_3(-1/u)$.
Taking the endpoint values (with the upper endpoint carrying the sign $+$ and the lower one $-$) and
summing over $j$, the coefficient of $\mathrm{Li}_3(-1/m)$ (before the overall factor $2$ of
\eqref{eq:KernelRep}) collects $-\rho_{m-2}+\rho_{m-1}$ from the first part and
$+\rho_{m-1}-\rho_m$ from the second, so that
\[
\tfrac12\,[\mathrm{Li}_3(-1/m)]_{B_k}=-\rho_{m-2}+2\rho_{m-1}-\rho_m=-\Delta^{2}\rho_{m-2},
\]
where $\Delta$ denotes the forward difference. By \eqref{eq:rhodiff},
$\rho_j-\rho_{j-1}=\frac{(-1)^{k+j-3}}{(k-1)!}j^{k-1}\binom{k-2}{j}$, hence
\[
-\Delta^{2}\rho_{m-2}
=\frac{(-1)^{k+m-2}}{(k-1)!}\Big[m^{k-1}\binom{k-2}{m}+(m-1)^{k-1}\binom{k-2}{m-1}\Big],
\]
which is \eqref{eq:Li3coeff}. Formula \eqref{eq:zeta3coeff} is the case $m=1$, using
$\mathrm{Li}_3(-1)=-\frac34\zeta(3)$.
\end{proof}

\begin{remark}[Where the new constants first occur]\label{rem:firstG}
The constants $G(-1/m)$ already appear in the cross integral for small $k$: for $k=3$ one has
$[G(-1/2)]_{B_3}=1$, and for $k=4$ one has $[G(-1/2)]_{B_4}=-\frac23$ and
$[G(-1/3)]_{B_4}=\frac83$ (equivalently $[G(-1/2)]_{E_{4,2}}=-4$ and $[G(-1/3)]_{E_{4,2}}=16$).
In the same decomposition, Proposition~\ref{prop:Li3} gives
$[\mathrm{Li}_3(-1/2)]_{B_3}=-1$, $[\mathrm{Li}_3(-1/2)]_{B_4}=\frac{10}{3}$ and
$[\mathrm{Li}_3(-1/3)]_{B_4}=-\frac83$; these are not the coefficients one reads off from the
standard-basis forms of Proposition~\ref{prop:Bk}, where $G(-1/2)$ and $G(-1/3)$ have been
replaced by their reductions below. For $k \leqslant 4$, however, these constants reduce to the
algebra generated by
$\mathfrak{L}_3$; for instance
\[
G\Big(-\frac12\Big)=\frac{13}{24}\zeta(3)+\mathrm{Li}_3\Big(-\frac12\Big),
\]
and $G(-1/3)$ has a similar reduction involving $\zeta(3)$, $\mathrm{Li}_3(-1/2)$,
$\mathrm{Li}_3(-1/3)$ and products of $\log2,\log3$ with $\mathrm{Li}_2(-1/2),\mathrm{Li}_2(-1/3)$.
Accordingly, Proposition~\ref{prop:Bk} needs no new constant for $k \leqslant 4$, and the first
constant that is \emph{not} reducible in this way is $G(-1/4)$, which enters at $k=5$ (with
coefficient $\frac{27}4$ in $B_5$, that is $\frac{135}2$ in $E_{5,2}$); its reduction involves
polylogarithms $\mathrm{Li}_3(\pm a/5)$, in accordance with item (b) of Remark~\ref{rem:Bk5}.
\end{remark}

\section{The Case \texorpdfstring{$k=2$}{k=2}: A Closed Formula for \texorpdfstring{$E_{2,n}$}{E2n}, with Numerical Verification for \texorpdfstring{$k=2,3$}{k=2,3}}\label{sec:E2n}

When $k = 2$, the entire sequence of coefficients $\{E_{2, n}\}_{n \geqslant 0}$ can be obtained in a single closed formula. In this case,
\[
h_n(\log t_1, \log t_2) = \frac{(\log t_1)^{n+1} - (\log t_2)^{n+1}}{\log t_1 - \log t_2}.
\]

\begin{theorem}[Closed formula for the $k = 2$ sequence]\label{thm:E2n}
For all integers $n \geqslant 0$,
\[
E_{2,n}=2(n+1)!\sum_{j=1}^{n+1}\frac{\eta(j)}{j}
=2(n+1)!\Big[\log2+\sum_{j=2}^{n+1}\frac{(1-2^{1-j})\zeta(j)}{j}\Big],
\]
where $\eta(s) = \sum_{m \geqslant 1} (-1)^{m-1} m^{-s}$ is the Dirichlet eta function ($\eta(1) = \log 2$ and $\eta(j) = (1 - 2^{1-j}) \zeta(j)$ for $j \geqslant 2$).
\end{theorem}

\begin{proof}
Make the exponential substitution $t_1 = \mathrm{e}^{-x}$, $t_2 = \mathrm{e}^{-y}$, which maps $(0, 1]^2$ onto $[0, \infty)^2$ with $\log t_1 = -x$, $\log t_2 = -y$, and $\mathrm{d}t_1\, \mathrm{d}t_2 = \mathrm{e}^{-(x + y)}\, \mathrm{d}x \mathrm{d}y$. By homogeneity $h_n(-x, -y) = (-1)^n h_n(x, y)$, the prefactor $(-1)^n$ cancels, giving
\[
E_{2, n} = \int_0^\infty \int_0^\infty \frac{h_n(x, y)}{\mathrm{e}^x + \mathrm{e}^y}\, \mathrm{d}x \mathrm{d}y.
\]
The integrand is symmetric in $x, y$, so we restrict to $0 < y < x$ and multiply by $2$. Substituting $h_n(x, y) = (x^{n+1} - y^{n+1})/(x - y)$ and $1/(\mathrm{e}^x + \mathrm{e}^y) = \mathrm{e}^{-x}/(1 + \mathrm{e}^{-(x - y)})$,
\[
E_{2, n} = 2 \int_0^\infty \mathrm{e}^{-x} \int_0^x \frac{x^{n+1} - y^{n+1}}{x - y} \cdot \frac{1}{1 + \mathrm{e}^{-(x - y)}}\, \mathrm{d}y\mathrm{d}x.
\]
Let $u = x - y$ (so $0 < u < x$) and interchange the order of integration (by Tonelli's theorem):
\[
E_{2, n} = 2 \int_0^\infty \frac{1}{u (1 + \mathrm{e}^{-u})} \left[\int_u^\infty \mathrm{e}^{-x} \big(x^{n+1} - (x - u)^{n+1}\big)\, \mathrm{d}x\right] \mathrm{d}u.
\]
The inner integral admits an explicit evaluation via the upper incomplete Gamma function:
\[
\int_u^\infty x^{n+1} \mathrm{e}^{-x}\, \mathrm{d}x = (n + 1)! \mathrm{e}^{-u} \sum_{\ell=0}^{n+1} \frac{u^\ell}{\ell!}, \qquad \int_u^\infty (x - u)^{n+1} \mathrm{e}^{-x}\, \mathrm{d}x = \mathrm{e}^{-u} (n + 1)!,
\]
whose difference is $\int_u^\infty \mathrm{e}^{-x} (x^{n+1} - (x - u)^{n+1})\, \mathrm{d}x = (n + 1)!\, \mathrm{e}^{-u} \sum_{\ell=1}^{n+1} u^\ell/\ell!$. Substituting back,
\[
E_{2, n} = 2 (n + 1)! \sum_{\ell=1}^{n+1} \frac{1}{\ell!} \int_0^\infty \frac{u^{\ell-1} \mathrm{e}^{-u}}{1 + \mathrm{e}^{-u}}\, \mathrm{d}u.
\]
The last integral is the Mellin representation of the Dirichlet eta function: $\int_0^\infty u^{s-1} \mathrm{e}^{-u}/(1 + \mathrm{e}^{-u})\, \mathrm{d}u = \Gamma(s) \eta(s)$ for $\Re s > 0$. Taking $s = \ell$ and using $\Gamma(\ell) = (\ell - 1)!$, this integral equals $(\ell - 1)!\, \eta(\ell)$, so
\[
E_{2,n}=2(n+1)!\sum_{\ell=1}^{n+1}\frac{(\ell-1)!}{\ell!}\eta(\ell)
=2(n+1)!\sum_{\ell=1}^{n+1}\frac{\eta(\ell)}{\ell}. \qedhere
\]
\end{proof}

By Theorem \ref{thm:E2n}, every coefficient in the complete expansion  \eqref{eq:mainexp} of $S_2(x)$ is an explicit rational combination of $\log 2$ and zeta values $\zeta(j)$; the first seven are listed in Table \ref{tab:E2n}.

\begin{table}[H]
\centering
\caption{Coefficients $E_{2, n}$ for $n = 0, \ldots, 6$}\label{tab:E2n}
\begin{tabular}{cl}
\toprule
$n$ & $E_{2,n}$ \\
\midrule
$0$ & $2\log2$ \\
$1$ & $\zeta(2)+4\log2$ \\
$2$ & $3\zeta(2)+3\zeta(3)+12\log2$ \\
$3$ & $12\zeta(2)+12\zeta(3)+\tfrac{21}{2}\zeta(4)+48\log2$ \\
$4$ & $60\zeta(2)+60\zeta(3)+\tfrac{105}{2}\zeta(4)+45\zeta(5)+240\log2$ \\
$5$ & $360\zeta(2)+360\zeta(3)+315\zeta(4)+270\zeta(5)+\tfrac{465}{2}\zeta(6)+1440\log2$ \\
$6$ & $2520\zeta(2)+2520\zeta(3)+2205\zeta(4)+1890\zeta(5)+\tfrac{3255}{2}\zeta(6)+\tfrac{2835}{2}\zeta(7)+10080\log2$ \\
\bottomrule
\end{tabular}
\end{table}

Recall that by Euler's formula, $\zeta(2j)$ is a rational multiple of $\pi^{2j}$ for every $j \geqslant 1$; in particular, $\zeta(2) = \pi^2/6$, $\zeta(4) = \pi^4/90$, and $\zeta(6) = \pi^6/945$.

\textbf{Numerical verification for $k = 2$}. The exact value of $S_2(x)$ is computed via $S_2(x) = \sum_s r_2(s)/s$, where $r_2(s) = \#\{(p, q) : p, q \leqslant x,\ p + q = s\}$ is the auto-convolution of the prime indicator function. Let the $m$-term approximation be
\[
\widehat{S}_2^{(m)}(x) = \frac{x}{\log^2 x} \sum_{n=0}^{m-1} \frac{E_{2, n}}{\log^n x},
\]
with coefficients $E_{2, 0}, \ldots, E_{2, m-1}$ given by Theorem \ref{thm:E2n}. Table \ref{tab:Snum} lists the values of $S_2(x)$, the $m = 5$ approximation and its absolute error, and the relative errors at each order $|\widehat{S}_2^{(m)}(x) - S_2(x)|/S_2(x)$, all rounded to two decimals (the sample points coincide with those of Figure \ref{fig:err}).

\begin{table}[htbp]
\centering
\caption{Values of $S_2(x)$ and of the $m = 5$ approximation, together with the absolute error of the latter and the relative errors of the $m$-term approximations for $m = 1, \dotsc, 5$, all rounded to two decimals (using the same $15$ sample points as Figure \ref{fig:err}).}\label{tab:Snum}
\begin{tabular}{rrrcccccc}
\toprule
\multirow{2}{*}{$x$} & \multirow{2}{*}{$S_2(x)$} & \multirow{2}{*}{$m=5$ approx.} & \multirow{2}{*}{Abs. error} & \multicolumn{5}{c}{Relative errors (\%)} \\
\cmidrule(l){5-9}
 & & & & $m=1$ & $m=2$ & $m=3$ & $m=4$ & $m=5$ \\
\midrule
499,999 & 5,443.16 & 5,431.50 & 11.66 & 26.05 & 8.09 & 2.87 & 1.01 & 0.21 \\
589,384 & 6,230.85 & 6,216.33 & 14.52 & 25.72 & 7.91 & 2.79 & 0.99 & 0.23 \\
694,747 & 7,131.65 & 7,117.72 & 13.93 & 25.36 & 7.68 & 2.66 & 0.92 & 0.20 \\
818,946 & 8,168.89 & 8,153.32 & 15.57 & 25.03 & 7.49 & 2.57 & 0.88 & 0.19 \\
965,348 & 9,363.06 & 9,343.51 & 19.55 & 24.73 & 7.33 & 2.51 & 0.87 & 0.21 \\
1,137,922 & 10,733.97 & 10,711.78 & 22.20 & 24.42 & 7.15 & 2.43 & 0.84 & 0.21 \\
1,341,347 & 12,310.27 & 12,285.26 & 25.02 & 24.12 & 6.98 & 2.35 & 0.81 & 0.20 \\
1,581,138 & 14,124.21 & 14,095.26 & 28.95 & 23.83 & 6.82 & 2.28 & 0.79 & 0.20 \\
1,863,796 & 16,210.03 & 16,177.96 & 32.07 & 23.54 & 6.66 & 2.20 & 0.76 & 0.20 \\
2,196,985 & 18,615.02 & 18,575.11 & 39.92 & 23.27 & 6.53 & 2.15 & 0.75 & 0.21 \\
2,589,737 & 21,376.82 & 21,334.95 & 41.87 & 22.98 & 6.37 & 2.07 & 0.71 & 0.20 \\
3,052,701 & 24,556.16 & 24,513.23 & 42.92 & 22.70 & 6.21 & 1.99 & 0.67 & 0.17 \\
3,598,428 & 28,221.32 & 28,174.37 & 46.95 & 22.43 & 6.06 & 1.92 & 0.64 & 0.17 \\
4,241,714 & 32,451.59 & 32,392.83 & 58.76 & 22.19 & 5.94 & 1.88 & 0.64 & 0.18 \\
4,999,999 & 37,327.11 & 37,254.67 & 72.44 & 21.95 & 5.83 & 1.84 & 0.63 & 0.19 \\
\bottomrule
\end{tabular}
\end{table}

The relative error decreases overall as $x$ increases, confirming \eqref{eq:mainexp}; over the large-$x$ range shown ($5 \times 10^5 \leqslant x \leqslant 5 \times 10^6$), the successive orders are clearly distinguished, with $m = 5$ already at approximately $0.2\%$ (the corresponding relative error curve is shown in Figure \ref{fig:err}). On the other hand, at a fixed small $x$, increasing the number of terms does not always improve accuracy: our calculations show that at $x = 10^3$, the optimal truncation is $m = 3$ (relative error $3.70\%$), after which $m = 4, 5$ in fact increase to $5.98\%$ and $13.82\%$. This is precisely the nature of an asymptotic (non-convergent) series: there exists an optimal truncation order depending on $x$.

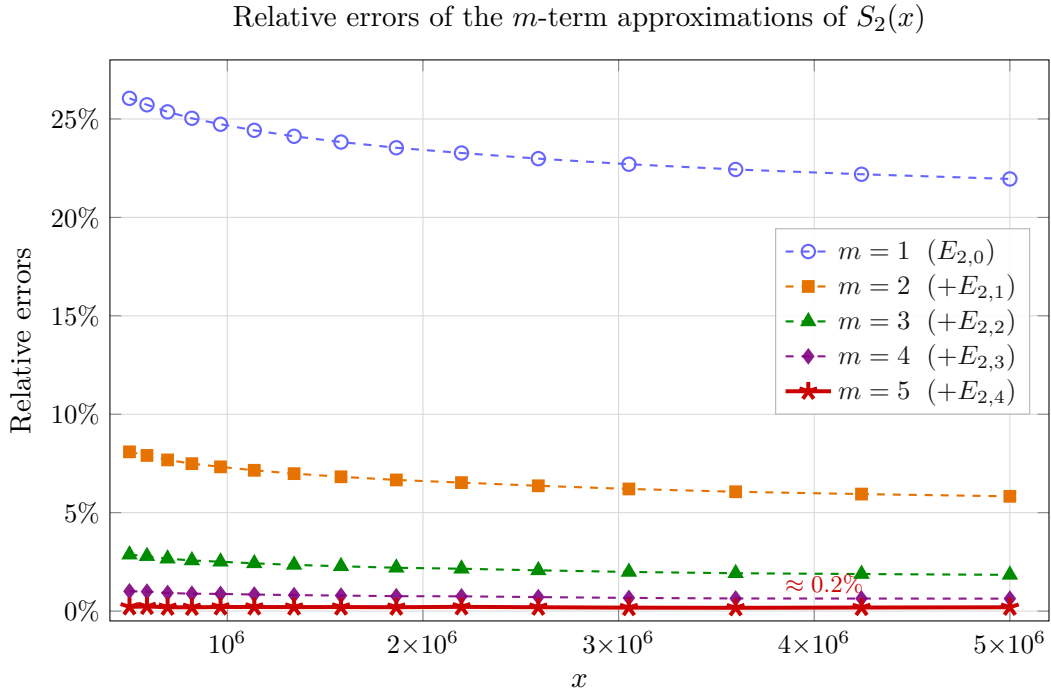
\begin{figure}[htbp]
\centering
\begin{tikzpicture}
\begin{axis}[
  width=14cm, height=9cm,
  xlabel={$x$},
  ylabel={Relative errors},
  title style={font=\normalsize},
  title={Relative errors of the $m$-term approximations of $S_2(x)$},
  xmin=400000, xmax=5200000,
  scaled x ticks=false,
  ymin=-0.005, ymax=0.28,
  xtick={1000000,2000000,3000000,4000000,5000000},
  xticklabels={$10^6$,$2{\times}10^6$,$3{\times}10^6$,$4{\times}10^6$,$5{\times}10^6$},
  ytick={0,0.05,0.10,0.15,0.20,0.25},
  yticklabels={$0$\%,$5$\%,$10$\%,$15$\%,$20$\%,$25$\%},
  grid=major,
  grid style={gray!30},
  legend style={at={(0.98,0.70)}, anchor=north east,
                font=\small, draw=gray!60, fill=white, fill opacity=0.92,
                cells={anchor=west}},
  tick label style={font=\small},
  label style={font=\normalsize},
  every axis plot/.append style={thick},
  clip=false,
]

% Order 1
\addplot[color=blue!60, mark=o, mark size=2.5pt, dashed, opacity=1,
         mark options={solid}]
  coordinates {
    (499999,0.260481)(589384,0.257214)(694747,0.253614)(818946,0.250343)
    (965348,0.247325)(1137922,0.244232)(1341347,0.241204)(1581138,0.238288)
    (1863796,0.235375)(2196985,0.232711)(2589737,0.229843)(3052701,0.227016)
    (3598428,0.224348)(4241714,0.221919)(4999999,0.219535)
  };
\addlegendentry{$m=1$\; ($E_{2,0}$)}

% Order 2
\addplot[color=orange!90!black, mark=square*, mark size=2pt, dashed, opacity=1,
         mark options={solid, fill=orange!90!black}]
  coordinates {
    (499999,0.080900)(589384,0.079073)(694747,0.076797)(818946,0.074896)
    (965348,0.073275)(1137922,0.071529)(1341347,0.069829)(1581138,0.068237)
    (1863796,0.066619)(2196985,0.065274)(2589737,0.063651)(3052701,0.062052)
    (3598428,0.060619)(4241714,0.059447)(4999999,0.058302)
  };
\addlegendentry{$m=2$\; ($+E_{2,1}$)}

% Order 3
\addplot[color=green!55!black, mark=triangle*, mark size=2.8pt, dashed, opacity=1,
         mark options={solid, fill=green!55!black}]
  coordinates {
    (499999,0.028673)(589384,0.027906)(694747,0.026632)(818946,0.025721)
    (965348,0.025073)(1137922,0.024264)(1341347,0.023475)(1581138,0.022771)
    (1863796,0.022012)(2196985,0.021515)(2589737,0.020702)(3052701,0.019889)
    (3598428,0.019227)(4241714,0.018816)(4999999,0.018411)
  };
\addlegendentry{$m=3$\; ($+E_{2,2}$)}

% Order 4
\addplot[color=violet, mark=diamond*, mark size=2.5pt, dashed, opacity=0.9,
         mark options={solid, fill=violet}]
  coordinates {
    (499999,0.010070)(589384,0.009906)(694747,0.009200)(818946,0.008840)
    (965348,0.008723)(1137922,0.008421)(1341347,0.008119)(1581138,0.007883)
    (1863796,0.007572)(2196985,0.007508)(2589737,0.007107)(3052701,0.006691)
    (3598428,0.006411)(4241714,0.006371)(4999999,0.006323)
  };
\addlegendentry{$m=4$\; ($+E_{2,3}$)}

% Order 5
\addplot[color=red!80!black, mark=star, mark size=3.5pt, solid, line width=1.5pt,
         mark options={solid, fill=red!80!black}]
  coordinates {
    (499999,0.002142)(589384,0.002330)(694747,0.001953)(818946,0.001906)
    (965348,0.002089)(1137922,0.002068)(1341347,0.002032)(1581138,0.002049)
    (1863796,0.001979)(2196985,0.002144)(2589737,0.001959)(3052701,0.001748)
    (3598428,0.001664)(4241714,0.001811)(4999999,0.001941)
  };
\addlegendentry{$m=5$\; ($+E_{2,4}$)}

\node[anchor=south west, font=\footnotesize, red!80!black] at (axis cs:3800000,0.004)
  {$\approx 0.2\%$};
\end{axis}
\end{tikzpicture}
\caption{Relative errors $|\widehat{S}_2^{(m)}(x) - S_2(x)|/S_2(x)$ of the $m$-term approximations of $S_2(x)$ for $m = 1, \ldots, 5$ over $5 \times 10^5 \leqslant x \leqslant 5 \times 10^6$. Each additional coefficient $E_{2, m-1}$ reduces the error by roughly a factor of three; at $m = 5$, the error has reached the $0.2\%$ level.}\label{fig:err}
\end{figure}

\textbf{Numerical verification for $k = 3$}. To confirm that the three-term expansion in Corollary \ref{cor:S3} behaves as predicted, we performed an analogous computation for $S_3(x)$. The exact value is obtained from $S_3(x) = \sum_s r_3(s)/s$, where
\[
r_3(s) = \#\{(p_1, p_2, p_3) : p_i \leqslant x,\ p_1 + p_2 + p_3 = s\} 
\]
is the triple auto-convolution of the prime indicator function; efficient evaluation uses two successive FFT convolutions. 

\begin{table}[H]
\centering
\caption{Values of $S_3(x)$ (computed exactly and rounded to two decimals) and relative errors of the $m$-term approximations at $m = 1, 2, 3$.}\label{tab:S3}
\begin{tabular}{rrccc}
\toprule
$x$ & $S_{3}(x)$ & $m=1$ & $m=2$ & $m=3$ \\
\midrule
$10^5$ & 7,582,338.45 & 32.17\% & 9.87\% & 2.81\% \\
$2 \times 10^5$ & 24,832,947.90 & 30.48\% & 8.93\% & 2.49\% \\
$5 \times 10^5$ & 121,482,963.27 & 28.52\% & 7.91\% & 2.18\% \\
$10^6$ & 408,231,672.07 & 27.09\% & 7.12\% & 1.85\% \\
$2 \times 10^6$ & 1,387,369,651.12 & 25.91\% & 6.58\% & 1.72\% \\
$5 \times 10^6$ & 7,074,866,381.38 & 24.43\% & 5.89\% & 1.51\% \\
$10^7$ & 24,455,712,935.40 & 23.36\% & 5.36\% & 1.29\% \\
\bottomrule
\end{tabular}
\end{table}

Table \ref{tab:S3} lists the values of $S_3(x)$ (rounded to two decimals), together with the relative errors of the $m$-term approximations
\[
\widehat{S}_3^{(m)}(x) = \frac{x^2}{\log^3 x} \sum_{n=0}^{m-1} \frac{E_{3, n}}{\log^n x} \qquad (m = 1, 2, 3),
\]
using the coefficients $E_{3, 0}, E_{3, 1}, E_{3, 2}$ from Corollary \ref{cor:S3}.

The pattern mirrors the $k = 2$ case: each successive coefficient reduces the relative error by roughly a factor of three, and at $x = 10^7$ the three-term approximation reaches the $1.3\%$ level. This provides direct numerical confirmation of the closed-form values of $E_{3, 0}$, $E_{3, 1}$, and $E_{3, 2}$ given in Corollary \ref{cor:S3}, and, by extension, of Proposition \ref{prop:Ek2} for the diagonal part and Proposition \ref{prop:Bk} for the cross part in the explicit case $k = 3$.

\section{Concluding Remarks}

\begin{remark}
The remainder in Theorem \ref{thm:main} is uniform for fixed $k$ and $N$; the parameter $\delta \in (0, 1)$ can be chosen arbitrarily, with the implied constant depending on $k, N, \delta$. The expansion is in general an asymptotic (non-convergent) series, consistent with typical asymptotic expansions: term-by-term valid but not necessarily globally convergent. Indeed, the substitution $t_j = \mathrm{e}^{-u_j}$ in \eqref{eq:Ekn} gives the sign-free representation $E_{k,n}=\int_{[0,\infty)^k}h_n(u_1,\dotsc,u_k)\,\mathrm{e}^{-(u_1+\dotsb+u_k)}(\mathrm{e}^{-u_1}+\dotsb+\mathrm{e}^{-u_k})^{-1}\,\mathrm{d}u \geqslant 0$, and since $h_n(u) \geqslant u_1^n$ and $\sum_j \mathrm{e}^{-u_j} \leqslant k$ one obtains $E_{k,n} \geqslant \frac{1}{k}\int_{[0,\infty)^k} u_1^n \mathrm{e}^{-(u_1+\dotsb+u_k)}\,\mathrm{d}u = \frac{n!}{k}$; hence $\sum_n E_{k,n}/\log^n x$ diverges for every fixed $x>1$. Note that this statement concerns the \emph{series} in $n$: every single coefficient is finite, and indeed Lemma~\ref{lem:conv} together with the lower bound above gives \begin{equation}\label{eq:two-sided} \frac{n!}{k}\;\leqslant\; E_{k,n}\;\leqslant\;\frac{(n+k-1)!}{k\,(k-1)!}\Big(\frac{k}{k-1}\Big)^{n+k}, \end{equation} so that the coefficients grow like $n!$ and the expansion, while valid to every fixed order, cannot converge; this is the usual behaviour of an asymptotic expansion. The numerical verifications in \S\ref{sec:E2n} (Tables \ref{tab:Snum} and \ref{tab:S3}) illustrate this concretely.
\end{remark}
\begin{remark}
The $\mathrm{Li}_2$ and $\mathrm{Li}_3$ coefficients of the diagonal part satisfy the proportionality $[\mathrm{Li}_2(-1/m)] = H_{k-1} [\mathrm{Li}_3(-1/m)]$, arising from $\beta_k = H_{k-1} \alpha_k$; this relation is broken in the total coefficient by the cross part. The reason is that, in $B_k$, the $\mathrm{Li}_3$ coefficients come only from the constant terms $P_j(0)$ of the kernel (whence Proposition~\ref{prop:Li3}), whereas the $\mathrm{Li}_2$ coefficients require the full higher coefficients of $P_j$; the two channels are therefore no longer proportional, and the deviation does not seem to admit a comparable closed form.
\end{remark}
\begin{remark}The higher-order coefficients $E_{k, n}$ ($n \geqslant 3$) have polylog order $n + 1$. Their diagonal parts can still be treated within the same generating-function framework (with the corresponding Laplace transforms raised to $\mathrm{Li}_{n+1}(-1/s)$ type), and for the cross part Proposition~\ref{prop:Gcoefficients} shows that, in the case $n=2$, the new constants $G(-1/m)$ occur with closed-form coefficients for every $k$; it would be interesting to know whether the analogous constants arising for $n \geqslant 3$ admit a similar uniform description.
\end{remark}

\begin{remark}[Computational aspects]\label{rem:computation}
Part of this work relies on computer-assisted computation, and we indicate here what is proved in
the paper and what is only verified or observed numerically.
\begin{enumerate}[label=(\roman*),leftmargin=2.2em]
\item The closed forms of $B_4$ and $E_{4,2}$ in \S\ref{sec:Ek2}, and the complete closed forms of
$E_{5,2},E_{6,2},E_{7,2}$ in \ref{app:Gcoeff}, were obtained by symbolic computation over
the rationals and verified independently against the defining integrals
$B_k=\int_0^\infty U(x)^2(1-\mathrm{e}^{-x})^{k-2}x^{-k}\,\mathrm{d}x$ and
$E_{k,2}=kA_k+\binom{k}{2}B_k$ by adaptive high-precision quadrature (working precision of $200$
decimal digits). The agreement is better than $50$ digits in every case reported. The coefficients
of $G(-1/m)$ are, however, not merely verified numerically: they are proved in \S\ref{sec:Gcoeff}.
The closed forms of Theorem \ref{thm:coeff}\,(ii) and of the diagonal part of
Proposition \ref{prop:Ek2} were also recomputed by exact rational arithmetic, evaluating
the finite differences of the generating functions of Lemmas \ref{lem:Hk} and \ref{lem:HA}
directly for $2\leqslant k\leqslant 8$; these recomputations, which bypass the term-by-term
bookkeeping of \S\ref{sec:Ek1} and \S\ref{sec:Ek2}, serve as a cross-check of those derivations.
\item The reductions of $G(-1/m)$ for $m=2,3,4$ and the negative statement of
Remark~\ref{rem:Bk5}\,(b) --- that the reduced constants do not appear to lie in
the $\mathbb{Q}$-algebra generated by $\mathfrak{L}_3$ --- rest on integer-relation detection (PSLQ)
applied to $200$--$460$ digit values. Such findings are numerical in nature, are stated as
observations only, and are not used as a proof anywhere in the paper.
\item The tables of \S\ref{sec:E2n} were computed from the exact values $S_k(x)=\sum_s r_k(s)/s$,
where the convolution counts $r_k(s)$ were evaluated by FFT and rounded to integers (the rounding
being checked against $\sum_s r_k(s)=\pi(x)^k$).
\end{enumerate}
All computations were carried out in \textsc{Python}~3.13 (with the libraries \textsc{NumPy},
\textsc{SciPy} and \textsc{mpmath}), and some of the symbolic reductions were additionally performed
independently in \textsc{Mathematica}~14.3. The scripts implementing these computations are
available as supplementary material to this paper.
\end{remark}

\appendix
\renewcommand{\thesection}{Appendix~\Alph{section}}
\section{Complete Closed Forms of \texorpdfstring{$E_{5,2}$}{E52}, \texorpdfstring{$E_{6,2}$}{E62} and \texorpdfstring{$E_{7,2}$}{E72}}\label{app:Gcoeff}

As an illustration of Proposition~\ref{prop:Gcoefficients}, we record the complete closed forms for
$k=5,6,7$. Combining
Proposition~\ref{prop:Gcoefficients} with the diagonal parts $kA_k$ gives the following
complete closed forms, in which
$G(-1/m)=\int_0^{-1/m}\mathrm{Li}_2(t)/(1+t)\,\mathrm{d}t$:

\medskip
\noindent\textbf{The case $k=5$.}
{\small
\begin{align*}
E_{5,2}={}&\tfrac52G\bigl(-\tfrac12\bigr)-40G\bigl(-\tfrac13\bigr)+\tfrac{135}{2}G\bigl(-\tfrac14\bigr)
-\tfrac{25}{8}\zeta(3)-\tfrac{20}{27}\pi^2\\
&-\tfrac{165}{2}\mathrm{Li}_3\bigl(-\tfrac12\bigr)+\tfrac{485}{2}\mathrm{Li}_3\bigl(-\tfrac13\bigr)
-\tfrac{1045}{6}\mathrm{Li}_3\bigl(-\tfrac14\bigr)\\
&-\tfrac{580}{3}\mathrm{Li}_2\bigl(-\tfrac12\bigr)+\tfrac{1245}{2}\mathrm{Li}_2\bigl(-\tfrac13\bigr)
-\tfrac{4310}{9}\mathrm{Li}_2\bigl(-\tfrac14\bigr)\\
&-\tfrac{8515}{27}\log2+\tfrac{45445}{32}\log3-\tfrac{56570}{27}\log4+\tfrac{1689775}{1728}\log5\\
={}&10.775787887755485099053547940801655930\ldots
\end{align*}}

\medskip
\noindent\textbf{The case $k=6$.}
{\small
\begin{align*}
E_{6,2}={}&-G\bigl(-\tfrac12\bigr)+48G\bigl(-\tfrac13\bigr)-243G\bigl(-\tfrac14\bigr)+256G\bigl(-\tfrac15\bigr)
+\tfrac98\zeta(3)+\tfrac{139}{480}\pi^2\\
&+\tfrac{2996}{15}\mathrm{Li}_2\bigl(-\tfrac12\bigr)-\tfrac{28467}{20}\mathrm{Li}_2\bigl(-\tfrac13\bigr)
+\tfrac{14412}{5}\mathrm{Li}_2\bigl(-\tfrac14\bigr)-\tfrac{41105}{24}\mathrm{Li}_2\bigl(-\tfrac15\bigr)\\
&+81\mathrm{Li}_3\bigl(-\tfrac12\bigr)-534\mathrm{Li}_3\bigl(-\tfrac13\bigr)
+1011\mathrm{Li}_3\bigl(-\tfrac14\bigr)-\tfrac{1137}{2}\mathrm{Li}_3\bigl(-\tfrac15\bigr)\\
&+\tfrac{71827}{225}\log2-\tfrac{1111031}{400}\log3+\tfrac{597713}{75}\log4
-\tfrac{2620739}{288}\log5+\tfrac{898441}{250}\log6\\
={}&11.4506921757156130275619734536452115\ldots
\end{align*}}

\medskip
\noindent\textbf{The case $k=7$.}
\begin{align*}
E_{7,2}={}&\tfrac{7}{24}G\bigl(-\tfrac12\bigr)-\tfrac{112}{3}G\bigl(-\tfrac13\bigr)
+\tfrac{1701}{4}G\bigl(-\tfrac14\bigr)-\tfrac{3584}{3}G\bigl(-\tfrac15\bigr)
+\tfrac{21875}{24}G\bigl(-\tfrac16\bigr)\\
&-\tfrac{1351}{24}\mathrm{Li}_3\bigl(-\tfrac12\bigr)+\tfrac{8953}{12}\mathrm{Li}_3\bigl(-\tfrac13\bigr)
-\tfrac{33775}{12}\mathrm{Li}_3\bigl(-\tfrac14\bigr)+\tfrac{94297}{24}\mathrm{Li}_3\bigl(-\tfrac15\bigr)
-\tfrac{218239}{120}\mathrm{Li}_3\bigl(-\tfrac16\bigr)\\
&-\tfrac{6538}{45}\mathrm{Li}_2\bigl(-\tfrac12\bigr)+\tfrac{164101}{80}\mathrm{Li}_2\bigl(-\tfrac13\bigr)
-\tfrac{370237}{45}\mathrm{Li}_2\bigl(-\tfrac14\bigr)+\tfrac{3491005}{288}\mathrm{Li}_2\bigl(-\tfrac15\bigr)
-\tfrac{293867}{50}\mathrm{Li}_2\bigl(-\tfrac16\bigr)\\
&-\tfrac{49}{160}\zeta(3)-\tfrac{7273}{86400}\pi^2
-\tfrac{790454}{3375}\log2+\tfrac{36023141}{9600}\log3-\tfrac{152001269}{8100}\log4\\
&+\tfrac{136474457}{3456}\log5-\tfrac{110672849}{3000}\log6+\tfrac{16391346181}{1296000}\log7\\
={}&12.2461808237389032054783939493146067214399561\ldots
\end{align*}
All coefficients above were computed by exact rational arithmetic and verified against the defining
integrals $E_{k,2}=kA_k+\binom{k}{2}B_k$ to more than $50$ digits; the coefficients of $G(-1/m)$ are
those of Proposition~\ref{prop:Gcoefficients}. The table below records, for $5\leqslant k\leqslant10$,
the number of constants $G(-1/m)$ occurring in $E_{k,2}$, the discrepancy between the two evaluations
of $B_k$ (closed form against numerical quadrature), and the value of $E_{k,2}$. The last column
records whether the corresponding complete closed form is displayed above or is obtained by the
same computation.

\begin{center}
\small
\begin{tabular}{ccccc}
\toprule
$k$ & \# of $G(-1/m)$ & $|B_k^{\text{closed}}-B_k^{\text{quad}}|$ & $E_{k,2}$ (first $20$ digits) & closed form\\
\midrule
$5$ & $3$ & $2.9\cdot10^{-59}$ & $10.77578788775548509905$ & given above\\
$6$ & $4$ & $6.6\cdot10^{-59}$ & $11.45069217571561302756$ & given above\\
$7$ & $5$ & $1.2\cdot10^{-58}$ & $12.24618082373890320548$ & given above\\
$8$ & $6$ & $8.0\cdot10^{-58}$ & $13.10541160763343464480$ & same computation\\
$9$ & $7$ & $5.1\cdot10^{-58}$ & $14.00248146026382765163$ & same computation\\
$10$ & $8$ & $3.9\cdot10^{-57}$ & $14.92388386023467489545$ & same computation\\
\bottomrule
\end{tabular}
\end{center}

The complete closed forms of $E_{8,2},E_{9,2},E_{10,2}$ are obtained by the same computation as for
$k\leqslant7$: the diagonal parts come from \S\ref{sec:Ek2}, while the cross parts are produced by
the Peano-kernel decomposition used in \S\ref{sec:Ek2}, whose $\mathrm{Li}_3$ and $G(-1/m)$
coefficients are given in closed form for all $k$ by Propositions~\ref{prop:Gcoefficients} and
\ref{prop:Li3}. Since the number of terms grows linearly with $k$, we record only the new
coefficients:

\begin{center}
\small
\begin{tabular}{c l}
\toprule
$k$ & $[G(-1/m)]_{E_{k,2}}$, $m=2,\dotsc,k-1$\\
\midrule
$8$ & $-\frac{1}{15},\ \frac{64}{3},\ -486,\ \frac{8192}{3},\ -\frac{15625}{3},\ \frac{15552}{5}$\\
$9$ & $\frac{1}{80},\ -\frac{48}{5},\ \frac{6561}{16},\ -4096,\ \frac{234375}{16},\ -\frac{104976}{5},\ \frac{823543}{80}$\\
$10$ & $-\frac{1}{504},\ \frac{32}{9},\ -\frac{2187}{8},\ \frac{40960}{9},\ -\frac{1953125}{72},\ 69984,\ -\frac{5764801}{72},\ \frac{2097152}{63}$\\
\bottomrule
\end{tabular}
\end{center}

\section*{Declarations}
\textbf{Funding:} This research received no external funding.
\textbf{Competing interests:} The authors declare no competing interests.
\textbf{Data availability:} All computational scripts used in this work are provided as supplementary material, and are also available from the authors on request.

\vspace{1em}
\noindent
Daoyi Peng, \textit{Email address}: \href{daoyi123@foxmail.com}{daoyi123@foxmail.com} \\
Hao Liu, \textit{Email address}: \href{liuhao@hbnu.edu.cn}{liuhao@hbnu.edu.cn}

\end{document}